\documentclass{amsart}
\usepackage[english]{babel}
\usepackage{amsmath,amsthm,amssymb,amsfonts}
\usepackage{hyperref}

\numberwithin{equation}{section}

\newtheorem{thm}{Theorem}[section]
\newtheorem{prp}[thm]{Proposition}
\newtheorem{lem}[thm]{Lemma}
\newtheorem{cor}[thm]{Corollary}
\theoremstyle{remark}

\newcommand{\tc}{\,:\,}                             

\newcommand{\id}{\mathrm{id}}                       
\newcommand{\chr}{\chi}                             
\DeclareMathOperator{\supp}{\mathrm{supp}}          

\newcommand{\R}{\mathbb{R}}                         
\newcommand{\N}{\mathbb{N}}                         
\newcommand{\C}{\mathbb{C}}                         

\newcommand{\D}{\mathcal{D}}                        
\newcommand{\E}{\mathcal{E}}                        
\newcommand{\Cv}{Cv}                                

\newcommand{\Diff}{\mathfrak{D}}                    

\newcommand{\LA}{\mathrm{L}}

\newcommand{\RA}{\mathrm{R}}

\newcommand{\lie}{\mathfrak}                        
\DeclareMathOperator{\Aut}{\mathrm{Aut}}            
\DeclareMathOperator{\tr}{\mathrm{tr}}              
\DeclareMathOperator{\Span}{\mathrm{span}}          

\newcommand{\Bdd}{\mathcal{B}}                      

\newcommand{\UEnA}{\mathrm{U}}                      

\newcommand{\Gelf}{\mathcal{G}}                     
\newcommand{\Kern}{\mathcal{K}}                     

\newcommand{\GS}{\mathfrak{G}}                      

\newcommand{\VV}{\mathcal{V}}                       
\newcommand{\HH}{\mathcal{H}}                       
\newcommand{\Alg}{\mathcal{A}}                      

\newcommand{\JJ}{\mathcal{J}}                       
\newcommand{\II}{\Gamma^1}                          
\newcommand{\IS}{\Gamma^2}                          
\newcommand{\PP}{\mathcal{P}}                       

\newcommand{\evmap}{\vartheta}
\renewcommand{\ell}{\varrho}

\begin{document}
\title[Spectral theory for algebras of differential operators]{Spectral theory for commutative algebras of differential operators on Lie groups}
\author{Alessio Martini}

\begin{abstract}
The joint spectral theory of a system of pairwise commuting self-adjoint left-invariant differential operators $L_1,\dots,L_n$ on a connected Lie group $G$ is studied, under the hypothesis that the algebra generated by them contains a ``weighted subcoercive operator'' of ter Elst and Robinson (J.\ Funct.\ Anal.\ 157 (1998) 88--163). The joint spectrum of $L_1,\dots,L_n$ in every unitary representation of $G$ is characterized as the set of the eigenvalues corresponding to a particular class of (generalized) joint eigenfunctions of positive type of $L_1,\dots,L_n$. Connections with the theory of Gelfand pairs are established in the case $L_1,\dots,L_n$ generate the algebra of $K$-invariant left-invariant differential operators on $G$ for some compact subgroup $K$ of $\Aut(G)$.
\end{abstract}

\maketitle

\section{Introduction}
Let $L_1,\dots,L_n$ pairwise commuting smooth linear differential operators on a smooth manifold $X$, which are formally self-adjoint with respect to some smooth measure $\mu$. Do these operators admit a joint functional calculus on $L^2(X,\mu)$? In that case, what is the relationship between the joint $L^2$ spectrum of $L_1,\dots,L_n$ and their joint smooth (possibly non-$L^2$) eigenfunctions on $X$?

A joint functional calculus for $L_1,\dots,L_n$ is given, via spectral integration, by a \emph{joint spectral resolution} $E$, i.e., a resolution of the identity of $L^2(X,\mu)$ on $\R^n$ such that
\[\int_{\R^n} \lambda_j \,dE(\lambda_1,\dots,\lambda_n)\]
is a self-adjoint extension of $L_j$ for $j=1,\dots,n$. Existence and uniqueness of $E$ are related to the so-called ``domain problems'', such as essential self-adjointness of $L_1,\dots,L_n$ and strong commutativity of their self-adjoint extensions.

Once a joint spectral resolution $E$ is fixed, the theory of eigenfunction expansions (see, e.g., \cite{berezanskii_expansions_1968,maurin_general_1968}) yields the existence, for $E$-almost every $\lambda = (\lambda_1,\dots,\lambda_n)$ in the joint $L^2$ spectrum $\Sigma = \supp E$ of $L_1,\dots,L_n$, of a corresponding generalized joint eigenfunction $\phi$, which (under some hypoellipticity hypothesis on $L_1,\dots,L_n$) belongs to the space $\E(X)$ of smooth functions on $X$ and satisfies
\begin{equation}\label{eq:eigenfunction}
L_j \phi = \lambda_j \phi \qquad\text{for $j=1,\dots,n$.}
\end{equation}
However, from the general theory, neither it is clear for which $\lambda \in \Sigma$ there does exist a corresponding smooth eigenfunction $\phi$, nor for which $\phi \in \E(X)$ satisfying \eqref{eq:eigenfunction} the corresponding $\lambda$ does belong to $\Sigma$.

In this paper, we restrict to the case of $X = G$ being a connected Lie group, with right Haar measure $\mu$, and left-invariant differential operators $L_1,\dots,L_n$. In this context, the problem of existence and uniqueness of a joint spectral resolution can be stated for the operators $d\varpi(L_1),\dots,d\varpi(L_n)$ in every unitary representation $\varpi$ of $G$ --- the case of the operators $L_1,\dots,L_n$ on $L^2(G)$ corresponding to the (right) regular representation of $G$ --- with a possibly different joint spectrum $\Sigma_\varpi$ for each representation $\varpi$.

Via techniques due to Nelson and Stinespring \cite{nelson_representation_1959}, we show in \S\ref{subsection:resolution} that a sufficient condition for the essential self-adjointness and the existence of a joint spectral resolution in every unitary representation is that the algebra generated by $L_1,\dots,L_n$ contains a \emph{weighted subcoercive operator}. This class of hypoelliptic left-invariant differential operators, defined by ter Elst and Robinson \cite{ter_elst_weighted_1998} in terms of a \emph{homogeneous contraction} of the Lie algebra $\lie{g}$ of $G$, is large enough to contain positive elliptic operators, sublaplacians and positive Rockland operators (see \S\ref{section:wsub} for details).

Under the same hypotheses on $L_1,\dots,L_n$, we prove that \emph{every} element of the joint spectrum $\Sigma$ corresponds to a joint (smooth) eigenfunction $\phi$ of $L_1,\dots,L_n$ which is a function of \emph{positive type} on $G$, i.e., of the form
\begin{equation}\label{eq:positivetypeintro}
\phi(x) = \langle \pi(x) v, v \rangle
\end{equation}
for some unitary representation $\pi$ of $G$ on a Hilbert space $\HH$ and some cyclic vector $v \in \HH \setminus \{0\}$. More precisely, in \S\ref{section:eigenfunctions} we show that:
\begin{enumerate}
\item[(a)] for every unitary representation $\varpi$ of $G$, $\Sigma_\varpi$ coincides with the set of the eigenvalues relative to the joint eigenfunctions of $L_1,\dots,L_n$ of the form \eqref{eq:positivetypeintro} with $\pi$ (irreducible and) \emph{weakly contained} in $\varpi$;
\item[(b)] if $G$ is amenable, then $\Sigma$ coincides with the set of the eigenvalues relative to \emph{all} the joint eigenfunctions of positive type;
\item[(c)] if $L^1(G)$ is a symmetric Banach $*$-algebra, then $\Sigma$ coincides with the set of the eigenvalues relative to all the \emph{bounded} joint eigenfunctions.
\end{enumerate}
Recall that, if $G$ has polynomial growth, then $L^1(G)$ is symmetric, and this in turn implies that $G$ is amenable (see \cite{palmer_banach_2001}). Notice moreover that, on non-amenable groups, the previous characterization (b) of $\Sigma$ cannot be expected, because of the \emph{spectral-gap} phenomenon (cf.\ \cite{varopoulos_hardy_1995}).

If there exists a compact group $K$ of automorphisms of $G$ such that the operators $L_1,\dots,L_n$ generate the algebra of left-invariant $K$-invariant differential operators on $G$, then the theory of \emph{Gelfand pairs} applies (see, e.g., \cite{gangolli_harmonic_1988,wolf_harmonic_2007}), and the joint spectral theory of $L_1,\dots,L_n$ is related to the spectral theory of the (convolution) algebra of $K$-invariant $L^1$ functions on $G$, i.e., to the \emph{spherical Fourier transform}. The ``Gelfand pair'' condition, however, is quite restrictive on the groups $G$ and the systems $L_1,\dots,L_n$ of operators which can be considered. Under our weaker hypotheses, we develop in \S\ref{section:algebras} a notion analogous to the spherical Fourier transform, with several similar features (Plancherel formula, Riemann-Lebesgue lemma, ...). Finally, in \S\ref{section:examples} some examples are considered, involving homogeneous groups and direct products, and moreover we show how (part of) the theory of Gelfand pairs on Lie groups fits in our general setting.

Some of the results presented here can be found in the literature in the case of a single operator ($n=1$), particularly for a sublaplacian (see, e.g., \cite{hulanicki_subalgebra_1974,hulanicki_commutative_1975,christ_multipliers_1991,ludwig_sub-laplacians_2000}), often as preliminaries for spectral multiplier theorems. It appears that our setting is suited for developing a theory of joint spectral multipliers for a family of commuting left-invariant differential operators on a Lie group (cf.\ \cite{martini_multipliers_2010}).

\subsection*{Notation}
For a topological space $X$, we denote by $C(X)$ the space of continuous (complex-valued) functions on $X$, whereas $C_0(X)$ and $C_c(X)$ are the subspaces of continuous functions vanishing at infinity and of continuous functions with compact support respectively. If $X$ is a smooth manifold, then $\E(X)$ and $\D(X)$ are the spaces of smooth functions and of compactly supported smooth functions on $X$; correspondingly, $\D'(X)$ and $\E'(X)$ are the spaces of distributions and of compactly supported distributions.

If $G$ is a Lie group, $f$ is a complex-valued function on $G$ and $x,y \in G$, then we set
\[\LA_x f(y) = f(x^{-1}y) , \qquad \RA_x f(y) = f(y x).\]
$\RA : x \mapsto \RA_x$ is the (right) regular representation of $G$. For a fixed right Haar measure $\mu$ on $G$, $\RA_x$ is an isometry of $L^p(G)$ for $1 \leq p \leq \infty$. With respect to such measure, convolution and involution of functions take the form
\[f * g(x) = \int_G f(xy^{-1}) g(y) \,dy, \qquad f^*(x) = \Delta(x) \overline{f(x^{-1})}\]
(where $\Delta$ is the modular function) and we set, for every representation $\pi$ of $G$,
\[\pi(f) = \int_G f(x) \, \pi(x^{-1}) \,dx,\]
so that in particular
\[\RA(g)f = f * g, \qquad \pi(f * g) = \pi(g) \pi(f), \qquad \pi(Df) = d\pi(D) \pi(f)\]
for every left-invariant differential operator $D$.

\section{Rockland and weighted subcoercive operators}\label{section:wsub}

This section is devoted to summarizing and amplifying some of the results of \cite{ter_elst_weighted_1998}, which are the basis for ours. In order to do this, however, it is useful first to recall some definitions and facts about homogeneous Lie groups; for more detailed expositions, we refer to the books \cite{folland_hardy_1982,goodman_nilpotent_1976,varadarajan_lie_1974}.

\subsection{Homogeneous groups and Rockland operators}
A \emph{homogeneous Lie algebra} is a Lie algebra $\lie{g}$ with a fixed family of automorphic dilations
\[\delta_t = e^{B \log t} \qquad \text{for $t > 0$,}\]
where $B$ is a diagonalizable derivation of $\lie{g}$ with strictly positive eigenvalues. The eigenspaces $W_\lambda$ of the derivation $B$ determine a direct-sum decomposition
\begin{equation}\label{eq:homogeneousalgebra}
\lie{g} = \bigoplus_{\lambda \in \R} W_\lambda = W_{\lambda_1} \oplus \dots \oplus W_{\lambda_k}
\end{equation}
(where $\lambda_k > \dots > \lambda_1 > 0$ are the eigenvalues of $B$) such that
\[[W_\lambda,W_{\lambda'}] \subseteq W_{\lambda+\lambda'} \qquad\text{for all $\lambda,\lambda' \in \R$.}\]
Every homogeneous Lie algebra $\lie{g}$ is \emph{nilpotent}, i.e., the \emph{descending central series}
\[\lie{g}_{[1]} = \lie{g}, \qquad \lie{g}_{[n+1]} = [\lie{g},\lie{g}_{[n]}]\]
is eventually null; in particular, $\lie{g}$ can be identified with the connected, simply connected Lie group $G$ whose Lie algebra is $\lie{g}$.

Let $G = \lie{g}$ be a homogeneous Lie group, with dilations $\delta_t = e^{B \log t}$. A \emph{homogeneous norm} on $G$ is a continuous function $|\cdot|_\delta : G \to \left[0,+\infty\right[$ such that 
\begin{itemize}
\item $|x|_\delta = 0$ if and only if $x$ is the identity of $G$;
\item $|x^{-1}|_\delta = |x|_\delta$;
\item $|\delta_t(x)|_\delta = t|x|_\delta$ for all $t > 0$.
\end{itemize}
Two homogeneous norms $|\cdot|_\delta$, $|\cdot|_\delta'$ on $G$ are always equivalent:
\[C^{-1} |x|_\delta \leq |x|_\delta' \leq C |x|_\delta \qquad\text{for all $x \in G$,}\]
for some constant $C \geq 1$ (see \cite{goodman_filtrations_1977}, \S3,  or \cite{goodman_nilpotent_1976}, \S1.2); moreover, there exists (see \cite{hebisch_smooth_1990}) a homogeneous norm $|\cdot|_\delta$ which is smooth off the origin and subadditive:
\[|xy|_\delta \leq |x|_\delta + |y|_\delta \qquad \text{for all $x,y \in G$.}\]
The quantity
\[Q_\delta = \tr B = \sum_{j=1}^k \lambda_j \dim W_j\]
is called the \emph{homogeneous dimension} of $\lie{g}$; in fact, we have
\[\mu(\delta_t(U)) = t^{Q_\delta} \mu(U)\]
for every measurable $U \subseteq \lie{g}$. Modulo rescaling (i.e., replacing $t$ with $t^c$ for some $c > 0$), one can suppose that $\lambda_1 \geq 1$, which shall be always understood in the rest of the paper, so that in particular $Q_\delta \geq \dim \lie{g}$.

The \emph{degree of polynomial growth} (or \emph{dimension at infinity}) of $G$ is the unique $Q_G \in \N$ such that
\[\mu(K^n) \sim n^{Q_G}\]
for every compact neighborhood $K = K^{-1}$ of the identity of $G$. This definition does not depend on the chosen dilations, and in fact it makes sense for every connected Lie group $G$ (with polynomial growth); for a nilpotent group $G$, we have the following characterization, where
\[\tau_K(x) = \min\{ n \in \N \tc x \in K^n \}.\]

\begin{prp}[Guivarc'h]\label{prp:nilpotentgrowth}
Suppose that $G$ is $s$-step nilpotent (i.e., $\lie{g}_{[s]} \neq 0 = \lie{g}_{[s+1]}$) and let $V_j$ be a complement of $\lie{g}_{[j+1]}$ in $\lie{g}_{[j]}$ for $j=1,\dots,s$. Choose moreover norms $|\cdot|_j$ on the $V_j$ and set
\begin{equation}\label{eq:growthnorm}
|x| = \sum_{j=1}^s |x_j|_j^{1/j},
\end{equation}
where $x = x_1 + \dots + x_s$ is the decomposition of $x \in \lie{g} = V_1 \oplus \dots \oplus V_s$. Then
\[|x| \sim \tau_K(x) \qquad\text{for large $x \in G$,}\]
for every compact neighborhood $K = K^{-1}$ of the identity. In particular, $G$ has polynomial growth of degree
\[Q_G = \sum_{j=1}^s j \dim V_j = \sum_{j=1}^s \dim \lie{g}_{[j]} \geq \dim \lie{g}.\]
\end{prp}
\begin{proof}
See \cite{guivarch_croissance_1973}, particularly the proofs of Th\'eor\`eme~II.1 and Lemme~II.1.
\end{proof}

A homogeneous Lie algebra $\lie{g}$ as in \eqref{eq:homogeneousalgebra} is \emph{stratified} if $W_1$ generates $\lie{g}$ as a Lie algebra (this implies that $\lambda_1,\dots,\lambda_k$ are integers). If $G = \lie{g}$ is stratified, then in Proposition~\ref{prp:nilpotentgrowth} one can take $V_j = W_j$, so that \eqref{eq:growthnorm} is a homogeneous norm on $G$ and $Q_G = Q_\delta$. For a general homogeneous Lie group, we have the following result (cf.\ also \cite{jenkins_dilations_1979}).

\begin{prp}\label{prp:quasiequivalence}
Let $G$ be a homogeneous Lie group, with dilations $\delta_t$ and homogeneous dimension $Q_\delta$, and let $|\cdot|_\delta$ be a homogeneous norm on $G$. Let $|\cdot|$ be defined as in \eqref{eq:growthnorm}, and $Q_G$ be the degree of polynomial growth of $G$.
\begin{itemize}
\item[(i)] One has $Q_\delta \geq Q_G$, with equality if and only if $G$ is stratified.
\item[(ii)] There exist $a,b,c > 0$ such that
\begin{equation}\label{eq:quasiequivalence}
c^{-1} |x|^a_\delta \leq |x| \leq c |x|_\delta^b \qquad\text{for $x \in G$ large}
\end{equation}
(i.e., off a compact neighborhood of the identity). Moreover, we can take $a = b = 1$ if and only if $G$ is stratified.
\end{itemize}
\end{prp}
\begin{proof}
(i) Decompose $\lie{g}$ as in \eqref{eq:homogeneousalgebra}. Notice that the subspaces $\lie{g}_{[n]}$ composing the descending central series are characteristic ideals of $\lie{g}$; since the dilations $\delta_t$ are automorphisms, the $\lie{g}_{[n]}$ are homogeneous. A homogeneous element of $\lie{g}_{[n]}$, being the sum of $n$-fold iterated commutators of homogeneous elements of $\lie{g}$, has a homogeneity degree which must be the sum of $n$ of the homogeneity degrees $\lambda_1 < \dots < \lambda_k$ of the elements of $\lie{g}$; since all these degrees are not less than $1$, the sum is not less than $n$, therefore $\lie{g}_{[n]} \cap W_\lambda = \{0\}$ if $\lambda < n$, so that
\begin{equation}\label{eq:homogeneouscontainment}
\lie{g}_{[n]} \subseteq \bigoplus_{\lambda \geq n} W_\lambda.
\end{equation}
In particular, if $G$ is $s$-step,
\begin{equation}\label{eq:homogeneousdimchain}
Q_G = \sum_{n = 1}^s \dim \lie{g}_{[n]} \leq \sum_{n=1}^s \sum_{\lambda \geq n} \dim W_\lambda \leq \sum_{j=1}^k \lfloor \lambda_j \rfloor \dim W_{\lambda_j} \leq Q_\delta.
\end{equation}

We already know that, if $G$ is stratified, then $Q_G = Q_\delta$. Conversely, if $Q_G = Q_\delta$, then all the inequalities in \eqref{eq:homogeneousdimchain} must be equalities; this means, first of all, that the degrees $\lambda_1,\dots,\lambda_k$ are integers and, secondly, that the inclusion \eqref{eq:homogeneouscontainment} is an equality, so that $W_n \subseteq \lie{g}_{[n]}$, but then necessarily $W_1$ generates $\lie{g}$ --- i.e., $G$ is stratified.

(ii) By the definition of $|\cdot|$ and the equivalence of homogeneous norms, the inequalities \eqref{eq:quasiequivalence} follow easily.

If $G$ is stratified, then also $|\cdot|_{\delta}$ is (modulo equivalence of homogeneous norms) of the form \eqref{eq:growthnorm}, with a choice of the complements $V_j$ possibly different to the one defining $|\cdot|$; therefore, by Proposition~\ref{prp:nilpotentgrowth}, $|\cdot|_\delta$ is equivalent in the large to $|\cdot|$ (both being equivalent in the large to some $\tau_K$). Conversely, since
\[\mu_G(\{x \in G \tc |x| < r\}) \sim r^{Q_G}, \qquad \mu_G(\{x \in G \tc |x|_\delta < r\}) \sim r^{Q_\delta}\]
for $r$ large, if \eqref{eq:quasiequivalence} holds with $a = b = 1$, then necessarily $Q_G = Q_\delta$, and the conclusion follows by (i).
\end{proof}

The automorphic dilations $\delta_t$ of a homogeneous Lie algebra $\lie{g}$ extend to automorphisms $\delta_t$ of its complex universal enveloping algebra $\UEnA(\lie{g})$, which is canonically isomorphic to the algebra $\Diff(G)$ of left-invariant differential operators on $G$. An element $D \in \UEnA(\lie{g}) = \Diff(G)$ is said to be homogeneous of degree $\lambda$ if
\[\delta_t(D) = t^\lambda D \qquad\text{for all $t > 0$.}\]

A \emph{Rockland operator} on $G$ is a homogeneous left-invariant differential operator $D \in \Diff(G)$ such that, for every non-trivial irreducible unitary representation $\pi$ of $G$ on a Hilbert space $\HH$, $d\pi(D)$ is injective on the space $\HH^\infty$ of the smooth vectors of the representation. In the abelian case ($G = \R^n$), with isotropic dilations, the notion of Rockland operator reduces to that of constant-coefficient homogeneous elliptic operator on $\R^n$. In the general case, by a theorem of Helffer and Nourrigat's (see \cite{helffer_caracterisation_1979}), combined with a result by Miller (see \cite{miller_parametrices_1980,ter_elst_spectral_1997}), a homogeneous $L \in \Diff(G)$ is Rockland if and only if $L$ is \emph{hypoelliptic}, i.e., for every $u \in \D'(G)$ and every open set $\Omega \subseteq G$,
\[(Lu)|_\Omega \in \E(\Omega) \quad\Longrightarrow\quad u|_\Omega \in \E(\Omega).\]

\subsection{Weighted bases and contraction of a Lie algebra}\label{subsection:noncommutativemultiindex}
A \emph{weighted (algebraic) basis} of a Lie algebra $\lie{g}$ is a system $A_1,\dots,A_d$ of linearly independent elements of $\lie{g}$ which generate $\lie{g}$ as a Lie algebra, together with the assignment of a weight $w_j \in \left[1,+\infty\right[$ to each $A_j$ ($j=1,\dots,d$).

Fix a weighted basis on $\lie{g}$. We recall some notation from \cite{ter_elst_weighted_1998}, analogous to the multi-index notation for partial derivatives on $\R^n$, but taking care of the non-commutative structure. Let $J(d)$ be the set of finite sequences of elements of $\{1,\dots,d\}$, and $J_+(d)$ be the subset of non-empty sequences. For every $\alpha = (\alpha_1,\dots,\alpha_k) \in J(d)$, let $|\alpha|$ denote the length $k$ of $\alpha$, and set $\|\alpha\| = \sum_{j=1}^k w_{\alpha_j}$,
\[A^\alpha = A_{\alpha_1} A_{\alpha_2} \cdots A_{\alpha_n} \text{ (as an element of $\UEnA(\lie{g})$),}\]
\[A_{[\alpha]} = [[\dots[A_{\alpha_1},A_{\alpha_2}],\dots],A_{\alpha_k}] \qquad\text{if $\alpha \in J_+(d)$.}\]

The fixed weighted basis defines an (increasing) filtration on $\lie{g}$:
\[F_\lambda = \Span \{A_{[\alpha]} \tc \alpha \in J_+(d), \, \|\alpha\| \leq \lambda\} \qquad \text{for $\lambda \in \R$;}\]
we have in fact
\[[F_{\lambda},F_{\mu}] \subseteq F_{\lambda+\mu}, \qquad F_\lambda = \bigcap_{\mu > \lambda} F_\mu, \qquad \bigcup_{\lambda \in \R} F_\lambda = \lie{g}.\]
Set $F_\lambda^- = \bigcup_{\mu < {\lambda}} F_\mu$; the weighted basis is said to be \emph{reduced} if\footnote{Our definition of reduced basis is more restrictive than the definition given in \S2 of \cite{ter_elst_weighted_1998}, where it is only required that $A_j \notin F_{w_j}^-$; however, without our restriction, the fundamental Lemma~2.2 of \cite{ter_elst_weighted_1998}, which allows to extend the reduced basis to a linear basis compatible with the associated filtration $F_\lambda$, is false, as it is shown by the following example. On the free $3$-step nilpotent Lie algebra on two generators, defined by
\[[X_1,X_2] = Y, \qquad [X_1,Y] = T_1, \qquad [X_2,Y] = T_2,\]
the weighted basis $X_1,X_2,Y+T_1,T_1,T_2$, with weights $1,1,3,3,3$, is reduced according to \cite{ter_elst_weighted_1998}, but it not compatible with the associated filtration, and cannot be extended since it is already a linear basis.}
\begin{equation}\label{eq:reducedbasis}
\Span \{ A_j \tc w_j = \lambda\} \cap F_\lambda^- = \{0\} \qquad\text{for all $\lambda$.}
\end{equation}
Given a weighted basis, it is always possible to remove some elements from it, in order to obtain a reduced basis of $\lie{g}$ which defines the same filtration. A \emph{weighted Lie algebra} is a Lie algebra with a fixed reduced (weighted) basis.

Notice that, if $\lie{g}$ is a homogeneous Lie algebra, every system of linearly independent generators $A_1,\dots,A_d$ of $\lie{g}$ made of homogeneous elements, with the weights equal to the respective homogeneity degrees, is a reduced basis of $\lie{g}$; such a basis is said to be \emph{adapted} to the homogeneous structure of $\lie{g}$. A \emph{weighted homogeneous Lie algebra} is a homogeneous Lie algebra with a fixed adapted basis.

Let $\lie{g}$ be a weighted Lie algebra, and let the filtration $(F_\lambda)_{\lambda}$ be defined as before. We can then consider the associated homogeneous Lie algebra (cf.\ \cite{bourbaki_lie1}, \S II.4.3): the filtration determines a finite set of weights $\lambda_1,\dots,\lambda_k$, with
\[1 \leq \lambda_1 < \dots < \lambda_k,\]
defined by the condition $F_{\lambda_j} \neq F_{\lambda_j}^{-}$ for $j=1,\dots,k$; if we put $W_\lambda = F_\lambda / F_\lambda^-$, then
\[\lie{g}_* = \bigoplus_{\lambda \in \R} W_\lambda = W_{\lambda_1} \oplus \dots \oplus W_{\lambda_k}\]
is a homogeneous Lie algebra, with weights $\lambda_1, \dots, \lambda_k$.

Since the fixed weighted basis $A_1,\dots,A_d$ is reduced, the corresponding weights $w_1,\dots,w_d$ are among the weights $\lambda_1,\dots,\lambda_k$ of the filtration; moreover, if $\bar A_j$ is the element of the quotient $W_{w_j}$ corresponding to $A_j \in F_{w_j}$, then $\bar A_1,\dots,\bar A_k$ is an adapted basis of $\lie{g}_*$, with the same weights $w_1,\dots,w_k$ (cf.\ \cite{ter_elst_weighted_1998}, Lemma~2.2 and Proposition~3.1). The homogeneous Lie algebra $\lie{g}_*$, with the fixed adapted basis $\bar A_1,\dots,\bar A_d$, is said to be the \emph{contraction} of the weighted Lie algebra $\lie{g}$.

Notice that, if $\lie{g}$ is a weighted homogeneous Lie algebra, then $\lie{g}_*$ is canonically isomorphic to $\lie{g}$.

A \emph{weighted Lie group} is a connected Lie group $G$ whose Lie algebra $\lie{g}$ is weighted. The \emph{contraction} $G_*$ of a weighted Lie group $G$ is the homogeneous Lie group whose Lie algebra is $\lie{g}_*$.

\subsection{Control distance and volume growth}\label{subsection:controldistance}

Let $G$ be a weighted Lie group. Let $A_1,\dots,A_k$ be the fixed reduced basis of its Lie algebra $\lie{g}$, with weights $w_1,\dots,w_k$. For $s \in \{0,\infty,*\}$ and $\varepsilon > 0$, let $C_s(\varepsilon)$ be the set of absolutely continuous arcs $\gamma : [0,1] \to G$ such that
\[\gamma'(t) = \sum_{j=1}^k \phi_j(t) \, A_j|_{\gamma(t)} \qquad\text{for a.e.\ $t \in [0,1]$,}\]
where
\begin{equation}\label{eq:mincondition}
|\phi_j(t)| < \begin{cases}
\varepsilon^{w_j} & \text{if $s = 0$,}\\
\varepsilon       & \text{if $s = \infty$,}\\
\min\{\varepsilon,\varepsilon^{w_j}\} & \text{if $s = *$,}
\end{cases} \qquad\text{for $t \in [0,1]$, $j=1,\dots,k$;}
\end{equation}
for $x,y \in G$, we define then
\[d_s(x,y) = \inf \{\varepsilon > 0 \tc \exists \gamma \in C_s(\varepsilon) \text{ with } \gamma(0) = x, \, \gamma(1) = y\}.\]

It is not difficult to show that $d_0$, $d_\infty$ and $d_*$ are left-invariant distances on $G$, compatible with the topology of $G$. In fact, $d_\infty$ is the classical ``unweighted'' Carnot-Carath\'eodory distance associated with the H\"ormander system $A_1,\dots,A_k$ (cf.\ \cite{varopoulos_analysis_1992}, \S III.4), while $d_0$ is a ``weighted'' Carnot-Carath\'eodory distance (similar to the ones studied in \cite{nagel_balls_1985}). Moreover, for $x,y \in G$, we have
\[d_0(x,y) \leq 1 \quad\iff\quad d_\infty(x,y) \leq 1 \quad\iff\quad d_*(x,y) \leq 1,\]
and the same holds with strict inequalities. Finally,
\[d_*(x,y) = \begin{cases}
d_0(x,y) &\text{for $d_*(x,y) \leq 1$,}\\
d_\infty(x,y) &\text{for $d_*(x,y) \geq 1$.}
\end{cases}\]
We call $d_*$ the \emph{control distance}\footnote{Notice that the definition of the control distance by ter Elst and Robinson in \S6 of \cite{ter_elst_weighted_1998} (see also \cite{ter_elst_weighted_1994}) is different from the one given here, and coincides with our distance $d_0$. Their definition has the advantage that, in the case of a homogeneous group with an adapted basis, the modulus $|\cdot|_0$ induced by $d_0$ is a homogeneous norm; on the other hand, this shows (by taking, e.g., any non-stratified homogeneous Lie group, cf.\ Propositions~\ref{prp:nilpotentgrowth} and \ref{prp:quasiequivalence}) that in general $d_0$ is not a ``connected distance'' as in \cite{varopoulos_analysis_1992}, \S III.4. Nevertheless, in the whole papers \cite{auscher_positive_1994,ter_elst_weighted_1994,ter_elst_weighted_1998} it is understood that $d_0$ is ``connected''.

By a careful examination of their proofs, one sees that the specific properties of $d_0$ are used only for small distances, whereas in the large only ``connectedness'' is used. Therefore, our modified definition of the control distance $d$ fixes the problem (as it has been confirmed to us by ter Elst in a private communication). As a side-effect, since $d_* \geq d_0$ everywhere, the heat kernel estimates obtained with this modification (see Theorem~\ref{thm:robinsonterelst}(e)) are stronger than the ones claimed by ter Elst and Robinson (which are therefore true \emph{a posteriori}).} on the weighted Lie group $G$.

The control distance $d_*$ induces a \emph{control modulus} $|\cdot|_*$ on $G$, given by
\[|g|_* = d_*(e,g).\]
Moreover, if $B_r$ denotes the $d_*$-ball with radius $r$ centered at the identity of $G$, then
\[\mu(B_r) \sim r^{Q_*} \qquad\text{for $r \leq 1$,}\]
where $Q_*$ is the homogeneous dimension of the contraction $\lie{g}_*$ (see \cite{ter_elst_weighted_1998}, Proposition~6.1). On the other hand, the growth rate of $\mu(B_r)$ for $r$ large coincides with the (intrinsic) volume growth of the group $G$ (cf.\ \cite{varopoulos_analysis_1992}, \S III.4); in particular, if $G$ has polynomial growth of degree $Q_G$, then
\[\mu(B_r) \sim r^{Q_G} \qquad\text{for $r \geq 1$.}\]

\subsection{Weighted subcoercive forms and operators}

Let $G$ be a weighted Lie group, with reduced basis $A_1,\dots,A_d$ of its Lie algebra $\lie{g}$, and weights $w_1,\dots,w_d$. In this context, a \emph{form} is an element of the free (non-commutative associative unital) algebra over $\C$ on $d$ indeterminates $X_1,\dots,X_d$; in other words, a form is a function $C : J(d) \to \C$ null off a finite subset of $J(d)$, which can be thought of as the non-commutative polynomial
\[\sum_{\alpha \in J(d)} C(\alpha) X^\alpha.\]
The \emph{degree} of the form $C$ is the number
\[\max \{\|\alpha\| \tc \alpha \in J(d),\,C(\alpha) \neq 0\}.\]
If $C$ is a form of degree $m$, then its \emph{principal part} is the form $P : J(d) \to \C$ which is given by the sum of the terms of $C$ of degree $m$:
\[P(\alpha) = \begin{cases}
C(\alpha) &\text{if $\|\alpha\| = m$,}\\
0 &\text{otherwise.}
\end{cases}\]
A form is said to be \emph{homogeneous} if it equals its principal part. The \emph{adjoint} of a form $C$ is the form $C^+$ defined by
\[C^+(\alpha) = (-1)^{|\alpha|} \overline{C(\alpha_*)},\]
where $\alpha_* = (\alpha_k,\dots,\alpha_1)$ if $\alpha = (\alpha_1,\dots,\alpha_k)$.

To each form $C$, we associate a differential operator $d\RA_G(C) \in \Diff(G)$ by setting
\[d\RA_G(C) = \sum_{\alpha \in J(d)} C(\alpha) A^\alpha.\]
More generally, if $\pi$ is a representation of $G$, we define
\[d\pi(C) = d\pi(d\RA_G(C)) = \sum_{\alpha \in J(d)} C(\alpha) d\pi(A)^\alpha.\]

Notice that we have
\[d\RA_G(C^+) = d\RA_G(C)^+,\]
where, for $D \in \Diff(G)$, $D^+$ denotes its \emph{formal adjoint} (with respect to the right Haar measure $\mu$), i.e., the element of $\Diff(G)$ determined by
\[\langle D f, g \rangle = \langle f, D^+ g \rangle \qquad\text{for all $f,g \in \D(G)$,}\]
where $\langle f, g \rangle = \int_G f \, \overline{g} \,d\mu$.

If $\pi$ is a representation of $G$ on a Banach space $\VV$, we define seminorms and norms on (subspaces of) $\VV$ by
\[N_{\pi,s}(x) = \max_{\substack{\alpha \in J(d)\\ \|\alpha\| = s}} \|dU(X^\alpha) x\|_\VV, \qquad \|x\|_{\pi,s} = \max_{\substack{\alpha \in J(d)\\ \|\alpha\| \leq s}} \|dU(X^\alpha) x\|_\VV,\]
for $s \in \R$, $s \geq 0$; these quantities are certainly defined on the space $\VV^\infty$ of smooth vectors of the representation. If $\pi$ is the right regular representation of $G$ on $L^p(G)$, we use the alternative notation $N_{p;s}$, $\|\cdot\|_{p;s}$ for the (semi)norms, and $L^{p;\infty}(G)$ for the space of smooth vectors.

A form $C$ of degree $m$ is said to be \emph{weighted subcoercive} on $G$ if $m/w_i \in 2\N$ for $i=1,\dots,d$ and if moreover the corresponding operator satisfies a local \emph{G\aa rding inequality}: there exist $\mu > 0$, $\nu \in \R$ and an open neighborhood $V$ of the identity $e \in G$ such that
\[\Re \langle \phi, d\RA_G(C) \phi\rangle \geq \mu (N_{2;m/2}(\phi))^2 - \nu \|\phi\|_2^2\]
for all $\phi \in \D(G)$ with $\supp \phi \subseteq V$. In this case, the operator $d\RA_G(C)$ is called a \emph{weighted subcoercive operator}.

Let $G_*$ be the contraction of $G$, with Lie algebra $\lie{g}_*$. Since $A_1,\dots,A_d$ induces a reduced basis $\bar A_1,\dots,\bar A_d$ on $\lie{g}_*$ (with the same weights), we can associate to a form $C$ both a differential operator $d\RA_G(C)$ on $G$ and a differential operator $d\RA_{G_*}(C)$ on $G_*$: in some sense, $d\RA_{G_*}(C)$ is the ``local counterpart'' of the operator $d\RA_G(C)$. The next theorem clarifies the relationship between the two operators.

\begin{thm}[ter Elst \& Robinson]\label{thm:robinsonterelst}
Let $C$ be a form of degree $m$, whose principal part is $P$, such that $m/w_i \in 2\N$ for $i=1,\dots,d$. The following are equivalent:
\begin{enumerate}
\item[(i)] $C$ is a weighted subcoercive form on $G$;
\item[(ii)] $d\RA_{G_*}(P + P^+)$ is a positive Rockland operator on $G_*$;
\item[(iii)] there are constants $\mu > 0$, $\nu \in \R$ such that, for every unitary representation $\pi$ of $G$ on a Hilbert space $\HH$,
\[\Re \langle x, d\pi(C) x \rangle \geq \mu \|x\|_{\pi,m/2}^2 - \nu \|x\|_\HH^2\]
for all $x \in \HH^\infty$;
\item[(iv)] there is a constant $\mu > 0$ such that, for every unitary representation $\pi$ of $G_*$ on a Hilbert space $\HH$,
\[\Re \langle x, d\pi(P) x \rangle \geq \mu (N_{\pi,m/2}(x))^2\]
for all $x \in \HH^\infty$.
\end{enumerate}
Moreover, if these conditions are satisfied, for every representation $\pi$ of $G$ on a Banach space $\VV$, we have:
\begin{itemize}
\item[(a)] the closure of $d\pi(C)$ generates a continuous semigroup $\{S_t\}_{t \geq 0}$ on $\VV$;
\item[(b)] for $t > 0$, $S_t(\VV) \subseteq \VV^\infty$, and moreover $\VV^\infty = \bigcap_{n = 1}^\infty D(\overline{d\pi(C)}^n)$;
\item[(c)] if $\pi$ is unitary, then $\overline{d\pi(C)} = d\pi(C^+)^*$;
\item[(d)] there exists a representation-independent kernel $k_t \in L^{1;\infty} \cap C_0^\infty(G)$ (for $t > 0$) such that
\[d\pi(X^\alpha) S_t x = \pi(A^\alpha k_t) x = \int_G (A^\alpha k_t)(g) \pi(g^{-1})x \,dg\]
for all $\alpha \in J(d)$, $t > 0$, $x \in \VV$;
\item[(e)] the kernel satisfies the following ``Gaussian'' estimates: for all $\alpha \in J(d)$ there exist $b,c,\omega > 0$ such that
\[|A^\alpha k_t(g)| \leq c t^{-\frac{Q_* + \|\alpha\|}{m}} e^{\omega t} e^{-b \left(\frac{|g|_*^m}{t}\right)^{1/(m-1)}}\]
for all $t >0$ and $g \in G$, where $Q_*$ is the homogeneous dimension of $\lie{g}_*$ and $|\cdot|_*$ is the control modulus;
\item[(f)] for all $\rho \geq 0$, the map $t \mapsto k_t$ is continuous $\left]0,+\infty\right[ \to L^{1;\infty}(G, e^{\rho |x|_*} \,dx)$ and, for all $\alpha \in J(d)$, there exist $c,\omega > 0$ such that
\[\|A^\alpha k_t\|_{L^1(G, e^{\rho |x|_*} \,dx)} \leq c t^{-\frac{\|\alpha\|}{m}} e^{\omega t};\]
\item[(g)] the function
\[k(t,x) = \begin{cases}
0 &\text{for $t \leq 0$,}\\
k_t(x) &\text{for $t > 0$,}
\end{cases}\]
on $\R \times G$ satisfies $\left( \frac{\partial}{\partial t} + d\RA_G(C) \right) k = \delta$ in the sense of distributions, where $\delta$ is the Dirac delta at the identity of $\R \times G$.
\end{itemize}
\end{thm}
\begin{proof}
This theorem is a summary of results contained in \cite{ter_elst_weighted_1998}, except for (f), since in Theorem~7.2 of \cite{ter_elst_weighted_1998} it is only stated that the map $t \mapsto k_t$ is continuous $\left]0,+\infty\right[ \to L^1(G, e^{\rho |x|_*} \,dx)$. However, the weighted $L^1$ estimates for $A^\alpha k_t$ in (f) are obtained by integration of the pointwise estimates (e), since the volume growth of a connected Lie group is at most exponential (cf.\ \cite{guivarch_croissance_1973}). Moreover, by the semigroup property, we have
\begin{equation}\label{eq:semigroupproperty}
A^{\alpha} (k_{t+s}) = k_t * (A^\alpha k_s)
\end{equation}
and, since $A^\alpha k_s \in L^1(G, e^{\rho |x|_*} \,dx)$, the required continuity follows from the properties of convolution.
\end{proof}

\begin{cor}\label{cor:hypoelliptic}
With the notation of the previous theorem, if $C$ is a weighted subcoercive form on $G$, then the function $k(t,x) = k_t(x)$ is smooth off the identity of $\R \times G$, and the operator $d\RA_G(C)$ is hypoelliptic.
\end{cor}
\begin{proof}
From Theorem~\ref{thm:robinsonterelst}(g) we deduce that, for every $r \in \N \setminus \{0\}$, the distribution
\begin{equation}\label{eq:derivativekernel}
(\partial_t^r - (- d\RA_G(C))^r) k
\end{equation}
is supported in the origin of $\R \times G$. In particular, if $\phi \in \D(\left]0,+\infty\right[)$ and $\psi \in \D(G)$, by applying \eqref{eq:derivativekernel} to $\phi \otimes \psi$ we get
\[(-1)^r \int_0^\infty \langle k_t, \psi \rangle \, \overline{\phi^{(r)}(t)} \,dt = \int_0^\infty \langle (-d\RA_G(C))^r k_t, \psi \rangle \, \overline{\phi(t)} \,dt.\]
Since both $t \mapsto k_t$ and $t \mapsto (-d\RA_G(C))^r k_t$ are continuous $\left]0,+\infty\right[ \to L^1(G)$ by Theorem~\ref{thm:robinsonterelst}(f), this identity holds also for all $\psi \in C_0(G)$. In other words, for all $\psi \in C_0(G)$, the $r$-th distributional derivative of the function $t \mapsto \langle k_t, \psi \rangle$ on $\left]0,+\infty\right[$ is the map
\[t \mapsto \langle (-d\RA_G(C))^r k_t, \psi \rangle;\]
since all these derivatives are continuous, the function $t \mapsto \langle k_t, \psi \rangle$ is smooth on $\left]0,+\infty\right[$, so that also the map $t \mapsto k_t$ is smooth $\left]0,+\infty\right[ \to L^1(G)$. But then from \eqref{eq:semigroupproperty} it follows easily that $t \mapsto k_t$ is smooth $\left]0,+\infty\right[ \to L^{1;\infty}(G)$. By Sobolev's embedding, we then get that $t \mapsto k_t$ is smooth $\left]0,+\infty\right[ \to \E(G)$; this gives that $k$ is smooth on $\left]0,+\infty\right[ \times G$, and the Gaussian estimates of Theorem~\ref{thm:robinsonterelst}(e) show that $k$ can be extended smoothly by zero to the whole $\R \times G \setminus \{(0,e)\}$.

Notice that $k_t^*$ is the kernel of $d\RA_G(C^+)$, which is also a weighted subcoercive operator. If we put
\[\tilde k(t,x) = \begin{cases}
0 &\text{if $t \geq 0$},\\
k_{-t}^* &\text{if $t \leq 0$,}
\end{cases}\]
then $\tilde k$ is smooth on $\R \times G \setminus \{(0,e)\}$ and satisfies $\left(-\frac{\partial}{\partial t} + d\RA_G(C^+)\right) \tilde k = \delta$
in the sense of distributions. By arguing analogously as in the proof of Theorem~52.1 of \cite{treves_topological_1967}, we obtain that $\partial_t + d\RA_G(C)$ is hypoelliptic on $\R \times G$, and the hypoellipticity of $d\RA_G(C)$ on $G$ follows immediately.
\end{proof}

\begin{cor}\label{cor:kernelapproximateidentity}
With the notation of Theorem~\ref{thm:robinsonterelst}, if $C$ is a weighted subcoercive form on $G$, then $(k_t)_{t > 0}$ is an \emph{approximate identity} on $G$ for $t \to 0^+$ (cf.\ \cite{grafakos_classical_2008}, \S1.2.4), i.e.,
\begin{itemize}
\item $k_t \in L^1(G)$ and $\limsup_{t \to 0^+} \|k_t \|_1 < \infty$;
\item $\lim_{t \to 0^+} \int_{G \setminus U} |k_t(x)| \,dx = 0$ for all neighborhoods $U$ of the identity of $G$;
\item $\lim_{t \to 0^+} \int_G k_t(x) \,dx = 1$.
\end{itemize}
More generally, for every $D \in \Diff(G)$, $\beta \geq 0$ and every neighborhood $U$ of the identity of $G$,
\begin{equation}\label{eq:kerneloutofaneighborhood}
\lim_{t \to 0^+} t^{-\beta} \int_{G \setminus U} |Dk_t(x)| \,dx = 0.
\end{equation}
\end{cor}
\begin{proof}
If $R > 0$ is such that
\[\{x \in G \tc |x|_* < R \} \subseteq U,\]
then, by Theorem~\ref{thm:robinsonterelst}(e), for $t \leq 1$ we have
\[t^{-\beta} \int_{G \setminus U} |Dk_t(x)| \,dx \leq c t^{-\gamma} \int_R^{+\infty} e^{-b(r^m/t)^{1/(m-1)}} e^{\sigma r} \,dr\]
for some $c,b,\sigma,\gamma > 0$. On the other hand, for $t \leq 1$ and $r \geq R$,
\[t^{-\gamma} e^{-b(r^m/t)^{1/(m-1)}} e^{\sigma r} \leq e^{-b(r^{\frac{m}{m-1}} - R^{\frac{m}{m-1}}) + \sigma r} e^{- \gamma \log t -b R^{\frac{m}{m-1}} t^{-\frac{1}{m-1}} },\]
where the first factor on the right-hand side is integrable on $\left]R,+\infty\right[$ and does not depend on $t$, whereas the second factor is infinitesimal for $t \to 0^+$ and does not depend on $r$; the limit \eqref{eq:kerneloutofaneighborhood} then follows by dominated convergence.

In particular, we have
\[\lim_{t \to 0^+} \int_{G \setminus U} |k_t(x)| \,dx = 0,\]
and moreover, by Theorem~\ref{thm:robinsonterelst}(f), the norms $\|k_t\|_1$ are uniformly bounded for $t$ small. Finally, if $\pi$ is the trivial representation of $G$ on $\C$ and if $c = d\pi(C) 1$, then by Theorem~\ref{thm:robinsonterelst}(d) we have
\[\int_G h_t(x) \,dx = \pi(h_t) 1 = e^{-tc},\]
which tends to $1$ as $t \to 0^+$.
\end{proof}

In the following, we will consider connected Lie groups $G$ with no previously fixed weighted structure; then, an operator $L \in \Diff(G)$ will be said \emph{weighted subcoercive} on $G$ if $L$ is weighted subcoercive with respect to some weighted structure on $\lie{g}$. In this sense, we can say that every positive Rockland operator on a homogeneous Lie group is weighted subcoercive (see \cite{ter_elst_spectral_1997}, Lemmata 2.2 and 2.4, and Theorem 2.5; see also \cite{ter_elst_weighted_1998}, Example 4.4). Moreover, it is easy to check that, for every choice of a system of linearly independent generators $A_1,\dots,A_d$ of a Lie algebra $\lie{g}$, the assignment of weights all equal to $1$ always gives a reduced basis, and that the corresponding contraction $\lie{g}_*$ is stratified; in particular, the sublaplacian $L = -(A_1^2 + \dots + A_d^2)$ is weighted subcoercive. Finally, if $A_1,\dots,A_d$ linearly generate $\lie{g}$, then the contraction $\lie{g}_*$ is Euclidean (abelian and isotropic), and it is not difficult to see that positive left-invariant elliptic operators on $G$ are weighted subcoercive with respect to this structure.

\section{Algebras of differential operators}\label{section:algebras}

Here the existence and uniqueness of a joint spectral resolution for a system $L_1,\dots,L_n$ of formally self-adjoint left-invariant differential operators on a connected Lie group $G$ is proved, under the hypothesis that the algebra generated by $L_1,\dots,L_n$ contains a weighted subcoercive operator. An analogue of the (inverse) spherical Fourier transform of Gelfand pairs is also defined, and its main properties are derived.

In this and the following sections, results from the theory of spectral integration (as presented, e.g., in \cite{berberian_notes_1966,rudin_functional_1973,fell_representations_1988}) will be used without further reference.

\subsection{Joint spectral resolution}\label{subsection:resolution}
In the following, $G$ will be a connected Lie group.

\begin{lem}\label{lem:quasirockland}
Let $D,L \in \Diff(G)$ and suppose that $L$ is weighted subcoercive and formally self-adjoint. Then, for some $\bar r \in \N$, we have that, for all $r \geq \bar r$, $L^r + D$ is weighted subcoercive.
\end{lem}
\begin{proof}
Fix a weighted structure on $\lie{g}$ with respect to which the operator $L$ is weighted subcoercive. Then there exists a weighted subcoercive form $C$ such that $d\RA_G(C) = L$, and also a form $B$ such that $d\RA_G(B) = D$. In fact, since $L^+ = L$, we can suppose that $C^+ = C$.

Let then $P$ be the principal part of $C$, so that, by Theorem~\ref{thm:robinsonterelst}, $d\RA_{G_*}(P)$ is Rockland. By definition, this implies that, for every $r \in \N \setminus \{0\}$, $P^r$ is Rockland too. Notice now that, if $r$ is sufficiently large so that $P^r$ has degree greater than that of $B$, then the principal part of $C^r + B$ is $P^r$ and this implies, by Theorem~\ref{thm:robinsonterelst} again, that $L^r + D = d\RA_G(C^r + B)$ is weighted subcoercive.
\end{proof}

For every $D \in \Diff(G)$ and every unitary representation $\pi$ of $G$ on a Hilbert space $\HH$, the operator $d\pi(D)$ will be considered as defined on the space $\HH^\infty$ of smooth vectors of $\pi$, and notions such as closure or essential self-adjointness are understood to be referred to this domain\footnote{For some particular representations $\pi$ one may be interested in considering other domains for the operators $d\pi(D)$: for instance, for the regular representation, one could consider the space $\D(G)$ of compactly supported smooth functions. Theorem~1.1 of \cite{nelson_representation_1959} shows that for this and other ``reasonable'' choices of the domain, the closure of the $d\pi(D)$ remains unvaried, thus results about essential self-adjointness do not change.}.

\begin{prp}\label{prp:commutativealgebra}
Let $\Alg$ be a commutative unital subalgebra of $\Diff(G)$ closed by formal adjunction and containing a weighted subcoercive operator. Then, for every unitary representation $\pi$ of $G$, we have
\begin{equation}\label{eq:determinazioneaggiunto}
\overline{d\pi(D)} = d\pi(D^+)^* \qquad\text{for all $D \in \Alg$;}
\end{equation}
moreover, the operators $\overline{d\pi(D)}$ for $D \in \Alg$ are normal and commute strongly pairwise.
\end{prp}
\begin{proof}
Let $L \in \Alg$ be weighted subcoercive. Since $\Alg$ is closed by formal adjunction, by replacing $L$ with $(L + L^+)/2$, we can suppose that $L$ is formally self-adjoint (see Theorem~\ref{thm:robinsonterelst}).

Let $D \in \Alg$. By Lemma~2.3 of \cite{nelson_representation_1959}, in order to prove \eqref{eq:determinazioneaggiunto} it is sufficient to show that $d\pi(D^+ D)$ is essentially self-adjoint. However, by Lemma~\ref{lem:quasirockland}, it is possible to find $r \in \N$ sufficiently large so that both $A = L^{2r}$ and $C = L^{2r} + D^+ D$ are weighted subcoercive, which implies by Theorem~\ref{thm:robinsonterelst}(c) that $d\pi(A)$ and $d\pi(C)$ are essentially self-adjoint. The conclusion that $d\pi(D^+ D) = d\pi(C) - d\pi(A)$ is essentially self-adjoint then follows as in the proof of Corollary~2.4 of \cite{nelson_representation_1959}.

From \eqref{eq:determinazioneaggiunto} it follows that, for every formally self-adjoint $D \in \Alg$, $d\pi(D)$ is essentially self-adjoint. Let now
\[\mathcal{Q} = \{D^2 \tc D = D^+ \in \Alg\}.\]
For all $A,B \in \mathcal{Q}$, we have that $A, B, (1+A)(1+B)$ are formally self-adjoint elements of $\Alg$, so that $d\pi(A), d\pi(B), d\pi((1+A)(1+B))$ are essentially self-adjoint, and moreover $d\pi(A + B + AB)$ is positive (notice that $AB \in \mathcal{Q}$); this implies, as in the proof of Corollary~2.4 of \cite{nelson_representation_1959}, that $\overline{d\pi(A)}$ and $\overline{d\pi(B)}$ commute strongly.

In order to conclude, it will be sufficient to show that every operator of the form $\overline{d\pi(D)}$ for some $D \in \Alg$ is the joint function of some of the operators $\overline{d\pi(A)}$ for $A \in \mathcal{Q}$. In fact, let $D = D_1 + i D_2$, where
\[D_1 = (D+D^+)/2,\qquad D_2 = (D-D^+)/{2i}\]
are both formally self-adjoint elements of $\Alg$. Then
\[D_1^2, (D_1+1/2)^2, D_2^2, (D_2+1/2)^2\]
are all elements of $\mathcal{Q}$, and we can consider the joint spectral resolution $E$ on $\R^4$ of the corresponding operators in the representation $\pi$. We then have, for $j=1,2$,
\[d\pi(D_j) = d\pi((D_j+1/2)^2 - D_j^2 - 1/4) \subseteq \int_{\R^4} f_j \,dE,\]
where $f_j(\lambda_{1,1},\lambda_{1,2},\lambda_{2,1},\lambda_{2,2}) = \lambda_{j,2}- \lambda_{j,1} - 1/4$, so that also
\[d\pi(D) \subseteq \int_{\R^4} (f_1 + if_2) \,dE, \qquad d\pi(D^+) \subseteq \int_{\R^4} (f_1 - if_2) \,dE;\]
by passing to the adjoints in the second inclusion and using \eqref{eq:determinazioneaggiunto}, we then get
\[\overline{d\pi(D)} = \int_{\R^4} (f_1 + if_2) \,dE,\]
and we are done.
\end{proof}

A system $L_1,\dots,L_n \in \Diff(G)$ will be called a \emph{weighted subcoercive system}\index{system of differential operators!weighted subcoercive} if $L_1,\dots,L_n$ are formally self-adjoint and pairwise commuting, and if moreover the unital subalgebra of $\Diff(G)$ generated by $L_1,\dots,L_n$ contains a weighted subcoercive operator. From the previous proposition and the spectral theorem we then have immediately

\begin{cor}\label{cor:commutativealgebra}
Let $L_1,\dots,L_n \in \Diff(G)$ be a weighted subcoercive system. For every unitary representation $\pi$ of $G$, the operators $\overline{d\pi(L_1)},\dots,\overline{d\pi(L_n)}$ admit a joint spectral resolution $E_\pi$ on $\R^n$ and, for every polynomial $p \in \C[X_1,\dots,X_n]$,
\begin{equation}\label{eq:closureform}
\overline{d\pi(p(L_1,\dots,L_n))} = \int_{\R^n} p \,dE_\pi.
\end{equation}
\end{cor}

In the following, the sign of closure for operators of the form \eqref{eq:closureform} for some weighed subcoercive system $L_1,\dots,L_n$ will be omitted.

\subsection{Kernel transform and Plancherel measure}\label{subsection:plancherel}
Let $G$ be a connected Lie group. We denote by $\Cv^2(G)$ the set of the distributions $k \in \D'(G)$ such that the operator $f \mapsto f * k$ is bounded on $L^2(G)$. By the Schwartz kernel theorem, there is a one-to-one correspondence between $\Cv^2(G)$ and the set of bounded linear operators $T$ on $L^2(G)$ which commute with left translations:
\[T \LA_x = \LA_x T \qquad\text{for all $x \in G$;}\]
thus we endow $\Cv^2(G)$ with the C$^*$-algebra structure of the latter. We then have the continuous embedding $L^1(G) \subseteq \Cv^2(G)$.

Let $L_1,\dots,L_n$ be a weighted subcoercive system on $G$. By applying Corollary~\ref{cor:commutativealgebra} to the (right) regular representation on $L^2(G)$, we obtain a joint spectral resolution $E$ of $L_1,\dots,L_n$. In particular, for every $f \in L^\infty(\R^n,E)$, we can consider the operator
\[f(L) = f(L_1,\dots,L_n) = E[f] = \int_{\R^n} f\, dE,\]
which is a bounded left-invariant linear operator on $L^2(G)$, so that it admits a kernel $\breve f \in \Cv^2(G)$:
\[f(L)u = u * \breve f \qquad\text{for all $u \in \D(G)$.}\]
In place of $\breve f$, we use also the notation $\Kern_L f$. The correspondence
\[\Kern_L : f \mapsto \Kern_L f\]
will be called the \emph{kernel transform} associated with the weighted subcoercive system $L_1,\dots,L_n$. The previous definitions and the properties of the spectral integral then yield immediately

\begin{lem}\label{lem:composition}
\begin{itemize}
\item[(a)] 
$\Kern_L$ is an isometric embedding of $L^\infty(\R^n,E)$ into $\Cv^2(G)$; in particular, for every $f \in L^\infty(\R^n,E)$,
\[\|\breve f\|_{\Cv^2} = \|f\|_{L^\infty(\R^n,E)}, \qquad \breve {\overline{f}} = (\breve f)^*.\]
\item[(b)] If $f, g \in L^\infty(\R^n,E)$ and $\breve g \in L^2(G)$, then
\[(fg)\breve{} = f(L) \breve g,\]
and in particular, if $\breve g \in \D(G)$, then
\[(fg)\breve{} = \breve g * \breve f.\]
\item[(c)] If $f, g \in L^\infty(\R^n,E)$, and if $g(\lambda) = \lambda_j f(\lambda)$ for some $j \in \{1,\dots,n\}$, then
\[\breve g = L_j \breve f\]
in the sense of distributions.
\end{itemize}
\end{lem}

The resemblance of $\Kern_L$ with an (inverse) Fourier transform goes beyond Lemma \ref{lem:composition}, and more refined properties of $\Kern_L$ follow from the fact that the algebra generated by $L_1,\dots,L_n$ contains a weighted subcoercive operator. In fact, we can find a polynomial $p_*$ with real coefficients such that $p_*(L)$ is weighted subcoercive; by replacing $p_*$ with $p_*^{2r}$ for some large $r \in \N$, we may suppose that $p_* \geq 0$ on $\R^n$ and that moreover, if we set
\[p_0(\lambda) = p_*(\lambda) + \sum_{j=1}^n \lambda_j^2 + 1,\]
\[p_k(\lambda) = p_0(\lambda) + \lambda_k \qquad\text{for $k=1,\dots,n$,}\]
then $p_0(L),p_1(L),\dots,p_n(L)$ are all weighted subcoercive (see Lemma~\ref{lem:quasirockland}). Notice that the polynomials $p_0,p_1,\dots,p_n$ are all strictly positive on $\R^n$ and
\[\lim_{\lambda \to \infty} p_k(\lambda) = +\infty \qquad\text{for $k=0,\dots,n$;}\]
moreover, $p_0(L),\dots,p_n(L)$ generate the same subalgebra of $\Diff(G)$ as $L_1,\dots,L_n$ do.

\begin{lem}\label{lem:Jdensity}
The subalgebra of $C_0(\R^n)$ generated by the functions
\[e^{-p_0}, e^{-p_1}, \dots, e^{-p_n}.\]
is a dense $*$-subalgebra of $C_0(\R^n)$.
\end{lem}
\begin{proof}
Since the functions $e^{-p_0}, e^{-p_1}, \dots, e^{-p_n}$ are real valued, the algebra generated by them is a $*$-subalgebra of $C_0(\R^n)$.

Notice that $e^{-p_0}$ is nowhere null. Moreover, if $\lambda,\lambda' \in \R^n$ and $\lambda \neq \lambda'$, then $\lambda_k \neq \lambda'_k$ for some $k \in \{1,\dots,n\}$, hence
\[\text{either}\quad e^{-p_0(\lambda)} \neq e^{-p_0(\lambda')} \quad\text{or}\quad e^{-p_k(\lambda)} \neq e^{-p_k(\lambda')}.\]
The conclusion then follows immediately by the Stone-Weierstrass theorem.
\end{proof}

Let now $\JJ_L$ be the subalgebra of $C_0(\R^n)$ generated by the functions of the form $e^{-q}$, where $q$ is a non-negative polynomial on $\R^n$ such that $q(L)$ is a weighted subcoercive operator on $G$ and $\lim_{\lambda \to \infty} q(\lambda) = +\infty$. Set moreover
\[C_0(L) = C_0(L_1,\dots,L_n) = \{\breve f \tc f \in C_0(\R^n)\}.\]
Finally, let $\Sigma$ be the joint spectrum\index{spectrum!joint spectrum} of $L_1,\dots,L_n$, i.e., the support of their joint spectral resolution $E$.

\begin{prp}\label{prp:Jdensity}
$C_0(L)$ is a sub-$C^*$-algebra of $\Cv^2(G)$, which is isometrically isomorphic to $C_0(\Sigma)$ via the kernel transform. Moreover
\[\Kern_L(\JJ_L) = \{\breve f \tc f \in \JJ_L\}\]
is a dense $*$-subalgebra of $C_0(L)$.
\end{prp}
\begin{proof}
For a function $f \in C_0(\R^n)$, we have
\[\|f\|_{L^\infty(\R^n,E)} = \sup_\Sigma |f| = \|f|_\Sigma \|_{C_0(\Sigma)}.\]
Since every $g \in C_0(\Sigma)$ extends to an $f \in C_0(\R^n)$ by the Tietze-Urysohn extension theorem, the first part of the conclusion follows immediately from Lemma~\ref{lem:composition}(a). The second part follows instead from Lemma~\ref{lem:Jdensity}.
\end{proof}

The results on weighted subcoercive operators and their heat kernels imply that the elements of $\Kern_L(\JJ_L)$ are particularly well-behaved. The next proposition, which shows a sort of commutativity between joint functional calculus of $L_1,\dots,L_n$ and unitary representations of $G$, is a multivariate analogue of Proposition~2.1 of \cite{ludwig_sub-laplacians_2000}.

\begin{prp}\label{prp:Jl1}
For every $f \in \JJ_L$, we have $\breve f \in L^{1;\infty}(G) \cap C^\infty_0(G)$ and moreover, for every unitary representation $\pi$ of $G$,
\[\pi(\breve f) = f(d\pi(L_1),\dots,d\pi(L_n)).\]
If $G$ is amenable, the last identity holds for every $f \in C_0(\R^n)$ with $\breve f \in L^1(G)$.
\end{prp}
\begin{proof}
Suppose first that $f$ is one of the generators $e^{-q}$ of $\JJ_L$. Then, by Corollary~\ref{cor:commutativealgebra} and the properties of the spectral integral,
\[e^{-q}(d\pi(L_1),\dots,d\pi(L_n)) = e^{-d\pi(q(L))},\]
and, since $q(L)$ is weighted subcoercive, we obtain from Theorem~\ref{thm:robinsonterelst}(d) that $\Kern_L(e^{-q}) \in L^{1;\infty} \cap C^\infty_0(G)$ and $e^{-q}(d\pi(L_1,\dots,L_n)) = \pi(\Kern_L(e^{-q}))$. The result is easily extended to every $f \in \JJ_L$ by Lemma~\ref{lem:composition}, the properties of convolution and those of the spectral integral.

Suppose now that $G$ is amenable, $f \in C_0(\R^n)$ and $\breve f \in L^1(G)$. By Proposition~\ref{prp:Jdensity}, we can find a sequence $f_j \in \JJ_L$ which converges uniformly to $f$ on $\R^n$. This implies in particular, by the properties of the spectral integral, that
\[f_j(d\pi(L_1),\dots,d\pi(L_n)) \to f(d\pi(L_1),\dots,d\pi(L_n))\]
in the operator norm, but also that $\breve f_j \to \breve f$ in $\Cv^2(G)$. Since $G$ is amenable, the representation $\pi$ is weakly contained in the regular representation (see \cite{greenleaf_invariant_1969}, \S3.5), so that also $\pi(\breve f_j) \to \pi(\breve f)$ in the operator norm. But then the conclusion follows immediately from the first part of the proof.
\end{proof}

We are now going to exploit the good properties of the kernels in $\Kern_L(\JJ_L)$ to obtain a Plancherel formula for the kernel transform $\Kern_L$. It should be noticed that, in the context of commutative Banach $*$-algebras, a general abstract argument yielding this kind of results is available (see \S26J of \cite{loomis_introduction_1953}, and also Theorem~1.6.1 of \cite{gangolli_harmonic_1988}). However, we believe that additional insight is provided by the explicit construction presented below, which follows essentially \cite{christ_multipliers_1991}, with some modifications due to our multivariate and possibly non-unimodular setting.

\begin{prp}\label{prp:compactlysupported}
If $f \in L^\infty(\R^n,E)$ is compactly supported, then
\[\breve f \in L^{2;\infty} \cap C^\infty_0(G).\]
\end{prp}
\begin{proof}
Let $\xi_t = e^{-tp_*}$ for $t > 0$, so that $\breve \xi_t \in L^{1;\infty}(G) \cap C_0^\infty(G)$.

Since $f$ is compactly supported, $f = g \,\xi_1$ with $g = f/\xi_1 \in L^\infty(\R^n,E)$, so that $\breve f = g(L) \breve \xi_{1} \in L^2(G)$
by Lemma~\ref{lem:composition}. Analogously, being $g$ compactly supported, also $\breve g \in L^2(G)$, but then $\breve f = \xi_1(L) \breve g = \breve g * \breve \xi_1 \in L^{2;\infty} \cap C^\infty_0(G)$, by Lemma~\ref{lem:composition} and properties of convolution.
\end{proof}

Thus we have plenty of kernels $\breve f$ which are in $L^2(G)$; as we are going see, the $L^2$-norm can be interpreted as an operator norm of a convolution operator. Let $\|\cdot\|_{\hat 2}$ denote the $L^2$ norm with respect to the left Haar measure $\Delta \mu$ (where $\Delta$ is the modular function), and correspondingly $\|\cdot\|_{\hat 2 \to \infty}$ the operator norm from $L^2(G,\Delta \mu)$ to $L^\infty(G)$; then it is easily shown that

\begin{lem}\label{lem:kernelnorm}
For all $f \in L^\infty(E)$, we have $\breve f \in L^2(G)$ if and only if
\[\|f(L)\|_{\hat 2 \to \infty} < \infty,\]
and in this case $\|\breve f\|_2 = \|f(L)\|_{\hat 2 \to \infty}$.
\end{lem}

We are now able to obtain a Plancherel formula for the kernel transform.

\begin{thm}\label{thm:plancherel}
The identity
\[\sigma(A) = \|E(A)\|_{\hat 2 \to \infty}^2 \qquad\text{for all Borel $A \subseteq \R^n$}\]
defines a regular Borel measure on $\R^n$ with support $\Sigma$, whose negligible sets coincide with those of $E$ and such that, for all $f \in L^\infty(E)$,
\[\int_{\R^n} |f|^2 \,d\sigma = \|f(L)\|_{\hat 2 \to \infty}^2 = \|\breve f\|_2^2.\]
\end{thm}
\begin{proof}
Clearly $\sigma(\emptyset) = 0$. Moreover, $\sigma$ is monotone: if $A \subseteq A'$ are Borel subsets of $\R^n$ and $\sigma(A') < \infty$, then, by Lemma~\ref{lem:kernelnorm}, $\breve\chr_{A'} \in L^2(G)$, so that, by Lemma~\ref{lem:composition}, also
\[\breve\chr_A = E(A) \breve\chr_{A'} \in L^2(G) \qquad\text{and}\qquad \|\breve \chr_A\|_2 \leq \|\breve\chr_{A'}\|_2,\]
i.e., $\sigma(A) \leq \sigma(A')$.

We now prove that $\sigma$ is finitely additive. Let $A, B \subseteq \R^n$ be disjoint Borel sets. By monotonicity, we may suppose that $\sigma(A),\sigma(B) < \infty$. Then, by Lemma~\ref{lem:kernelnorm}, both $\breve\chr_{A},\breve\chr_{B} \in L^2(G)$, but
\[E(A \cup B) = E(A) + E(B),\]
so that clearly $\breve\chr_{A \cup B} = \breve\chr_A + \breve\chr_B \in L^2(G)$, and moreover, by Lemma~\ref{lem:kernelnorm},
\[\sigma(A \cup B) = \|\breve\chr_{A \cup B}\|_2^2 = \|\breve\chr_A\|_2^2 + \|\breve\chr_B\|_2^2 = \sigma(A) + \sigma(B),\]
since $\breve\chr_A = E(A)\breve\chr_A \perp E(B)\breve\chr_B = \breve\chr_B$ in $L^2(G)$ by Lemma~\ref{lem:composition}.

Finite additivity implies that, if $A_j$ ($j \in \N$) are pairwise disjoint Borel subsets of $\R^n$ and $A = \bigcup_j A_j$, then
\[\sum_j \sigma(A_j) \leq \sigma(A).\]
In particular, if the sum on the left-hand side diverges, then we have an equality. Suppose instead that the left-hand side sum converges. Then, as before, the $\breve\chr_{A_j}$ are pairwise orthogonal elements of $L^2(G)$, and their sum converges in $L^2(G)$ to some $k \in L^2(G)$ such that $\|k\|_2^2 = \sum_j \sigma(A_j)$. But then, if $u \in \D(G)$, we have that, on one hand, by Lemma~\ref{lem:kernelnorm},
\[\sum_j u * \breve\chr_{A_j} = u * k \qquad\text{uniformly,}\]
and, on the other hand,
\[\sum_j u * \breve\chr_{A_j} = \sum_j E(A_j) u = E(A)u \qquad\text{in $L^2(G)$,}\]
which gives, by uniqueness of limits and arbitarity of $u \in \D(G)$,
\[\breve\chr_A = k \in L^2(G) \qquad\text{and}\qquad \sigma(A) = \|k\|_2^2 = \sum_j \sigma(A_j).\]

It is immediate from the definition that a Borel subset of $\R^n$ is $\sigma$-negligible if and only if it is $E$-negligible; in particular $\supp \sigma = \supp E = \Sigma$.

By Proposition~\ref{prp:compactlysupported}, $\sigma(A) = \|\chr_A(L)\|_{\hat 2 \to \infty}^2 = \|\breve\chr_A\|_2^2$ is finite if $A \subseteq \R^n$ is relatively compact. We can then conclude, by Theorem~2.18 of \cite{rudin_real_1974}, that $\sigma$ is regular.

Notice that, for all Borel $A \subseteq \R^n$ with $\sigma(A) < \infty$, $\sigma$ coincides with the measure $\langle E(\cdot) \breve \chr_A, \breve \chr_A\rangle$ on the subsets of $A$: in fact, for all Borel $B \subseteq \R^n$,
\[\langle E(B) \breve\chr_A , \breve\chr_A \rangle = \|\breve\chr_{A \cap B}\|_2^2 = \sigma(A \cap B)\]
by Lemmata~\ref{lem:kernelnorm} and \ref{lem:composition}. In particular, for all $f \in L^\infty(E)$ with $\supp f \subseteq A$,
\[\int_{\R^n} |f|^2 \,d\sigma = \int_{\R^n} |f(\lambda)|^2 \,\langle E(d\lambda) \breve\chr_A, \breve\chr_A\rangle = \|f(L) \breve\chr_A\|_2^2 = \|\breve f\|_2^2 = \|f(L)\|_{\hat 2 \to \infty}^2\]
by the properties of the spectral integral and Lemmata~\ref{lem:kernelnorm} and \ref{lem:composition}.

Take now a countable partition of $\R^n$ made of relatively compact Borel subsets $A_j$ ($j \in \N$). Then, for every $f \in L^\infty(\R^n,E)$, analogously as before we obtain\
\[\|f(L)\|_{\hat 2 \to \infty}^2 = \sum_j \|E(A_j) f(L)\|_{\hat 2 \to \infty}^2 = \sum_j \|\Kern_L(f \chr_{A_j})\|_2^2,\]
and putting all together we get the conclusion.
\end{proof}

The measure $\sigma$ of the previous proposition is called the \emph{Plancherel measure} associated with the system $L_1,\dots,L_n$. Notice that
\[L^\infty(\R^n,E) = L^\infty(\sigma).\]

We now show that the estimates (for small times) on the heat kernel of weighted subcoercive operators give information on the behaviour at infinity of the Plancherel measure. In the following $|\cdot|_2$ shall denote the Euclidean norm.

\begin{prp}\label{prp:plancherelpolynomialgrowth}
The Plancherel measure $\sigma$ on $\R^n$ associated with a weighted subcoercive system $L_1,\dots,L_n$ has (at most) polynomial growth at infinity.
\end{prp}
\begin{proof}
If $\xi_t(\lambda) = e^{-t p_*(\lambda)}$, then, for every $r > 0$,
\[\sigma(\{p_* \leq r\}) = \|\chr_{\{p_* \leq r\}}\|_{L^2(\sigma)}^2 \leq e^2 \|\xi_{1/r}\|_{L^2(\sigma)}^2 = e^2 \|\breve \xi_{1/r}\|_{L^2(G)}^2.\]
Since $\breve \xi_t$ is the heat kernel of the operator $p_*(L_1,\dots,L_n)$, Theorem~\ref{thm:robinsonterelst}(e,f) gives, for large $r$,
\[\sigma(\{p_* \leq r\}) \leq C r^{Q_*/m},\]
where $m$ is the degree of $p_*(L_1,\dots,L_n)$ with respect to a suitable reduced weighted algebraic basis of $\lie{g}$, and $Q_*$ is the homogeneous dimension of the corresponding contraction $\lie{g}_*$. In particular, if $d$ is the degree of the polynomial $p_*$, we get, for large $a > 0$,
\[\sigma(\{\lambda \tc |\lambda|_2 \leq a\}) \leq \sigma(\{p_* \leq C(1 + a)^d\}) \leq C (1+a)^{Q_* d/m},\]
which is the conclusion.
\end{proof}

The proof of Proposition~\ref{prp:plancherelpolynomialgrowth} shows that the degree of growth at infinity of the Plancherel measure $\sigma$ is somehow related to the ``local dimension'' $Q_*$ of the group with respect to the control distance associated with the chosen weighted subcoercive operator (see \S\ref{subsection:controldistance}). In \S\ref{subsection:homogeneity} we will obtain more precise information on the behaviour of $\sigma$ under the hypothesis of homogeneity.

By Theorem~\ref{thm:plancherel}, $\Kern_L|_{L^2 \cap L^\infty(\sigma)}$ extends to an isometry from $L^2(\sigma)$ onto a closed subspace of $L^2(G)$. We give now an alternative characterization of this subspace. Namely, let $\IS_L$ be the closure of $\Kern_L(\JJ_L)$ in $L^2(G)$.

\begin{prp}\label{prp:Jl2}
$\Kern_L|_{L^2 \cap L^\infty(\sigma)}$ extends to an isometric isomorphism
\[L^2(\sigma) \to \IS_L.\]
\end{prp}

In fact, this result follows immediately from Theorem~\ref{thm:plancherel} and the following

\begin{lem}\label{lem:Jlq}
$\JJ_L$ is dense in $L^q(\sigma)$ for $1 \leq q < \infty$.
\end{lem}
\begin{proof}
Since $\sigma$ has polynomial growth at infinity (see Proposition~\ref{prp:plancherelpolynomialgrowth}), whereas the elements of $\JJ_L$ decay exponentially, it is easily seen that $\JJ_L$ is contained (modulo restriction to $\Sigma$) in $L^1 \cap L^\infty(\sigma)$. Since $\sigma$ is a positive regular Borel measure on $\R^n$, in order to prove that the closure of $\JJ_L$ in $L^q(\sigma)$ is the whole $L^q(\sigma)$, it is sufficient to show that $C_c(\R^n)$ is contained in this closure (see \cite{rudin_real_1974}, Theorem~3.14).

Let then $m \in C_c(\R^n)$. By Lemma~\ref{lem:Jdensity}, we can find a sequence $m_k \in \JJ_L$ converging uniformly to $m$, so that $\sup_k \|m_k\|_\infty = C < \infty$. Thus, for every $t > 0$, $m_k e^{-tp_0}$ converges uniformly to $m e^{-tp_0}$, dominated by $ Ce^{-tp_0} \in L^q(\sigma)$, and consequently $m_k e^{-tp_0} \to m e^{-tp_0}$ also in $L^q(\sigma)$; we then have that $m e^{-tp_0}$ is in the closure of $\JJ_L$ in $L^q(\sigma)$ for all $t > 0$, and by monotone convergence also $m$ is in this closure.
\end{proof}

We now prove a sort of Riemann-Lebesgue lemma for $\Kern_L^{-1}$.

\begin{prp}\label{prp:riemannlebesgue1}
For every bounded Borel $f : \R^n \to \C$ with $\breve f \in L^1(G)$, we have
\[\|f\|_{L^\infty(\sigma)} \leq \|\breve f\|_1,\]
and moreover
\[\lim_{r \to +\infty} \| f \, \chr_{\{\lambda \tc |\lambda|_2 \geq r\}} \|_{L^\infty(\sigma)} = 0.\]
\end{prp}
\begin{proof}
The first inequality follows immediately from Lemma~\ref{lem:composition} and Young's inequality.

Let $\xi_t = e^{-t p_0}$. Then, by Corollary~\ref{cor:kernelapproximateidentity}, $\breve\xi_t$ is an approximate identity for $t \to 0^+$. In particular, if $\breve f \in L^1(G)$, then
\[\Kern_L(f \xi_t) = \breve f * \breve \xi_t \to \breve f \qquad\text{in $L^1(G)$}\]
for $t \to 0^+$, which implies, by the first inequality, that
\[\lim_{t \to 0^+} \|f (1 - \xi_t) \|_{L^\infty(\sigma)} = 0.\]
Therefore, for every $\varepsilon > 0$, there exists $t > 0$ such that $\|f (1 - \xi_t) \|_{L^\infty(\sigma)} \leq \varepsilon$; since $p_0(\lambda) \to +\infty$ for $\lambda \to \infty$, we may find $r > 0$  such that
\[\|\xi_t \, \chr_{\{\lambda \tc |\lambda|_2 \geq r\}}\|_\infty \leq 1/2,\]
but then necessarily $\|f \, \chr_{\{\lambda \tc |\lambda|_2 \geq r\}}\|_\infty \leq 2\varepsilon$.
\end{proof}

An analogous (and neater) result for $\Kern_L$ is obtained under the additional hypothesis of unimodularity.

\begin{prp}\label{prp:riemannlebesgue2}
If $G$ is unimodular and $f \in L^1 \cap L^\infty(\sigma)$, then $\breve f \in C_0(G)$ and
\[\|\breve f\|_\infty \leq \|f\|_{L^1(\sigma)}.\]
\end{prp}
\begin{proof}
Since $f \in L^1 \cap L^\infty(\sigma)$, for some Borel $g_1,g_2 : \R^n \to \C$ we have
\[f = g_1 g_2 \qquad\text{and}\qquad |g_1|^2 = |g_2|^2 = |f|;\] in particular, $g_1,g_2 \in L^2 \cap L^\infty(\sigma)$. Therefore $\breve g_1, \breve g_2 \in L^2(G)$ by Theorem \ref{thm:plancherel} and
\[\breve f = \breve g_1 * \breve g_2\]
by Lemma \ref{lem:composition}, which gives the conclusion by Young's inequality (see \cite{hewitt_abstract_1979}, Theorem~20.16).
\end{proof}

\subsection{Change of generators}\label{section:automorphisms}

Let $L_1,\dots,L_n$ be a weighted subcoercive system on a connected Lie group $G$. Let $\sigma$ be the associated Plancherel measure on $\R^n$, and $\Sigma = \supp \sigma$. For given polynomials $P_1,\dots,P_{n'} : \R^n \to \R$, consider the operators
\[L'_1 = P_1(L_1,\dots,L_n), \quad\dots,\quad L'_{n'} = P_k(L_1,\dots,L_n),\]
and suppose that they still form a weighted subcoercive system. Let $\sigma'$ be the Plancherel measure on $\R^{n'}$ associated with the system $L'_1,\dots,L'_{n'}$, and $\Sigma'$ its support. We may ask if there is a relationship between the transforms $\Kern_{L}$ and $\Kern_{L'}$, and between the Plancherel measures $\sigma$ and $\sigma'$.

Let $P : \R^n \to \R^{n'}$ denote the polynomial map whose $j$-th component is the polynomial $P_j$.

\begin{lem}\label{lem:properwsub}
The map $P|_\Sigma : \Sigma \to \R^{n'}$ is a proper continuous map.
\end{lem}
\begin{proof}
Since $L_1',\dots,L_{n'}'$ is a weighted subcoercive system, we can find a non-negative polynomial $Q : \R^{n'} \to \R$ such that $Q(L') = Q(P(L))$ is a weighted subcoercive operator. By Theorem~\ref{thm:robinsonterelst}(iii), for sufficiently large $C > 0$ and $k \in \N$, we have that
\[\max_j \|L_j \phi\|_2 \leq C \|(1 + Q(P(L))^k) \phi\|_2 \qquad\text{for $\phi \in \D(G)$,}\]
which means, by the spectral theorem, that
\[\max_j |\lambda_j| \leq C(1 + Q(P(\lambda))^k) \qquad\text{for $\lambda \in \Sigma$,}\]
since $\Sigma$ is the joint spectrum of $L_1,\dots,L_n$.

Now, if $K \subseteq \R^{n'}$ is compact, then by continuity there exists $M > 0$ such that $Q|_K \leq M$, but then
\[\max_j |\lambda_j| \leq C(1 + M^k) \qquad\text{for $\lambda \in \Sigma \cap P^{-1}(K)$,}\]
thus $P^{-1}(K) \cap \Sigma$ is bounded in $\R^n$, and also closed (by continuity of $P$), therefore $P^{-1}(K)$ is compact.
\end{proof}

\begin{prp}\label{prp:pushforward}
For every bounded Borel $m : \R^{n'} \to \C$, we have:
\[m(L') = (m \circ P)(L), \qquad \Kern_{L'} m = \Kern_{L} (m \circ P).\]
Moreover
\[\sigma' = P(\sigma), \qquad \Sigma' = P(\Sigma).\]
\end{prp}
\begin{proof}
The first part of the conclusion follows immediately from the spectral theorem and uniqueness of the convolution kernel. From this, the identity $\sigma' = P(\sigma)$ is easily inferred by Theorem~\ref{thm:plancherel}. In particular,
\[\sigma(\R^n \setminus P^{-1}(\Sigma')) = \sigma'(\R^{n'} \setminus \Sigma') = 0,\]
i.e., by continuity of $P$, $P(\Sigma) \subseteq \Sigma'$.

In order to prove the opposite inclusion, we use the fact that $P|_\Sigma$ is proper (see Lemma~\ref{lem:properwsub}). Take $\lambda' \in \Sigma'$, and let $B_k$ be a decreasing sequence of compact neighborhoods of $\lambda'$ in $\R^{n'}$ such that $\bigcap_k B_k = \{\lambda'\}$. By definition of support, we then have $\sigma(P^{-1}(B_k)) = \sigma'(B_k) \neq 0$, therefore $P^{-1}(B_k) \cap \Sigma \neq \emptyset$ for all $k$. Since $P|_\Sigma$ is proper, we have a decreasing sequence $P^{-1}(B_k) \cap \Sigma$ of non-empty compacta of $\R^n$, which therefore has a non-empty intersection. If $\lambda$ belongs to this intersection, then clearly $\lambda \in \Sigma$ and moreover $P(\lambda) \in B_k$ for all $k$, that it, $P(\lambda) = \lambda'$.
\end{proof}

A particularly interesting case is when $L_1',\dots,L_{n'}'$ generate the same subalgebra of $\Diff(G)$ as $L_1,\dots,L_n$. In this case, there exists also a polynomial map $Q = (Q_1,\dots,Q_n) : \R^{n'} \to \R^n$ such that
\[L_1 = Q_1(L'), \quad\dots, \quad L_n = Q_n(L').\]
Notice that in general $P$ and $Q$ are not the inverse one of the other: from the spectral theorem, we only deduce that $(Q \circ P)|_\Sigma = \id_\Sigma$, $(P \circ Q)|_{\Sigma'} = \id_{\Sigma'}$ (in fact, these identities extend to the Zariski-closures of $\Sigma$ and $\Sigma'$). In particular,
\[P|_\Sigma : \Sigma \to \Sigma', \qquad Q|_{\Sigma'} : \Sigma' \to \Sigma\]
are homeomorphisms.

Another way of producing new weighted subcoercive systems from a given one is via the action of automorphisms of $G$. Namely, if $k \in \Aut(G)$, then its derivative $k'$ is an automorphism of $\lie{g}$, therefore it extends to a unique filtered $*$-algebra automorphism of $\Diff(G) \cong \UEnA(\lie{g})$ (which shall be still denoted by $k'$), and clearly
\begin{equation}\label{eq:newwsubsystem}
k'(L_1),\dots,k'(L_n)
\end{equation}
is a weighted subcoercive system on $G$. Notice that, for every $k \in \Aut(G)$, the push-forward via $k$ of the right Haar measure $\mu$ on $G$ is a multiple of $\mu$, and in fact there is a Lie group homomorphism $c : \Aut(G) \to \R^+$ such that
\[k(\mu) = c(k) \mu.\]
In particular, if we set
\[T_k f = f \circ k^{-1}\]
for $k \in \Aut(G)$, then the properties of the spectral integral and those of convolution give immediately

\begin{prp}\label{prp:automorphismskernel}
For $k \in \Aut(G)$, $T_k$ is a multiple of an isometry of $L^2(G)$; more precisely
\[\|T_k f\|_2^2 = c(k)^{-1} \|f\|_2^2.\]
Moreover, for all $D \in \Diff(G)$,
\[k'(D) = T_k D T_k^{-1}.\]
In particular, for every bounded Borel $m : \R^n \to \C$,
\[m(k'(L_1),\dots,k'(L_n)) = T_k m(L_1,\dots,L_n) T_k^{-1},\]
and consequently
\[\Kern_{k'(L)} m = c(k) T_k \Kern_L m.\]
\end{prp}

Let $\mathcal{O}$ be the unital subalgebra of $\Diff(G)$ generated by $L_1,\dots,L_n$. For any automorphism $k \in \Aut(G)$, we say that $\mathcal{O}$ is $k$-invariant if $k(\mathcal{O}) \subseteq \mathcal{O}$, or equivalently, if $k(\mathcal{O}) = \mathcal{O}$ (the equivalence is due to the fact that $k'$ is an injective linear map preserving the filtration of $\Diff(G)$, which is made of finitely dimensional subspaces).

Let $\Aut(G;\mathcal{O})$ denote the (closed) subgroup of $\Aut(G)$ made of the automorphisms $k$ such that $\mathcal{O}$ is $k$-invariant. If $k \in \Aut(G;\mathcal{O})$, then \eqref{eq:newwsubsystem} must be a system of generators of $\mathcal{O}$; therefore, we can choose a polynomial map $P_k = (P_{k,1},\dots,P_{k,n}) : \R^n \to \R^n$ such that $k'(L_j) = P_{k,j}(L)$. Hence, by putting together Propositions~\ref{prp:pushforward} and \ref{prp:automorphismskernel}, we get

\begin{cor}\label{cor:automorphismskernel}
If $k \in \Aut(G;\mathcal{O})$, then, for every bounded Borel $m : \R^n \to \C$,
\[(m \circ P_k)(L_1,\dots,L_n) = T_k m(L_1,\dots,L_n) T_k^{-1}\]
and
\[\Kern_L (m \circ P_k) = c(k) T_k \Kern_L m.\]
Moreover,
\[P_k(\sigma) = c(k) \sigma, \qquad P_k(\Sigma) = \Sigma.\]
\end{cor}

In particular, the restrictions $P_k|_\Sigma$ (which are univocally determined by $k$) define an action of the group $\Aut(G;\mathcal{O})$ on the spectrum $\Sigma$ by homeomorphisms; more precisely

\begin{prp}
The map
\begin{equation}\label{eq:automorphismsaction}
\Aut(G;\mathcal{O}) \times \Sigma \ni (k,\lambda) \mapsto P_{k^{-1}}(\lambda) \in \Sigma
\end{equation}
is continuous, and defines a continuous (left) action of $\Aut(G;\mathcal{O})$ on $\Sigma$.
\end{prp}
\begin{proof}
Recall that $\Sigma$ may be identified, as a topological space, with the Gelfand spectrum of the sub-C$^*$-algebra $C_0(L)$ of $\Cv^2(G)$, where $\lambda \in \Sigma$ corresponds to the multiplicative linear functional $\psi_\lambda$ defined by $\psi_\lambda(\breve m) = m(\lambda)$. By Corollary~\ref{cor:automorphismskernel} we then deduce
\[\psi_{P_k(\lambda)} = c(k) \psi_\lambda \circ T_k,\]
which clarifies that \eqref{eq:automorphismsaction} defines a left action on $\Sigma$. Moreover, since $C_0(L) \cap L^1(G)$ is dense in $C_0(L)$ (see Proposition~\ref{prp:Jdensity}), and since $c(k)T_k$ is an isometry of $\Cv^2(G)$, we obtain easily that $k \mapsto c(k) T_k u$ is continuous for every $u \in \Cv^2(G)$. Therefore, since the topology of the Gelfand spectrum is induced by the weak-$*$ topology, we immediately obtain that \eqref{eq:automorphismsaction} is separately continuous, and also jointly continuous since the $\psi_\lambda$ have uniformly bounded norms.
\end{proof}

In conclusion, the richer the group $\Aut(G;\mathcal{O})$ is, the more we may deduce about the structure of the spectrum $\Sigma$ and the Plancherel measure $\sigma$. An example of this fact is illustrated in \S\ref{subsection:homogeneity}.

\section{Spectrum and eigenfunctions}\label{section:eigenfunctions}
Let $L_1,\dots,L_n$ be a weighted subcoercive system on a connected Lie group $G$. We keep the notation of \S\ref{subsection:plancherel}. Notice that every $m \in \JJ_L$ is real analytic and admits a unique holomorphic extension to $\C^n$, which we still denote by $m$.

\begin{prp}\label{prp:weakstrongeigenfunctions}
Let $\phi \in \D'(G)$ be such that, for some $\lambda = (\lambda_1,\dots,\lambda_n) \in \C^n$,
\[L_j \phi = \lambda_j \phi \qquad\text{for $j=1,\dots,n$}\]
in the sense of distributions. Then $\phi \in \E(G)$, and the previous equalities hold in the strong sense. Moreover, if $\phi \in L^\infty(G)$, then, for every $m \in \JJ_L$,
\begin{equation}\label{eq:eigenconvolution}
\phi * \breve m = m(\lambda) \phi \qquad\text{and}\qquad \langle \breve m, \phi \rangle = m(\lambda) \, \overline{\phi(e)}.
\end{equation}
\end{prp}
\begin{proof}
From the hypothesis, we get immediately
\[p_*(L) \phi = p_*(\lambda) \phi.\]
Since $p_*(L) - p_*(\lambda)$ is hypoelliptic by Corollary~\ref{cor:hypoelliptic}, this implies that $\phi \in \E(G)$.

Suppose now that $\phi$ is bounded. Let $e^{-q}$ be one of the generators of $\JJ_L$, and set $k_t = \Kern_L(e^{-tq})$. Then, for every $x \in G$, also $\LA_x \phi$ is a joint eigenfunction of $L_1,\dots,L_n$ with eigenvalue $\lambda$; therefore, by Theorem~\ref{thm:robinsonterelst}(f,g), the function
\[t \mapsto \phi * k_t(x) = \langle \LA_x \phi, k_t \rangle\]
is smooth on $\left]0,+\infty\right[$, with derivative
\[t \mapsto \langle \LA_x \phi, -q(L) k_t \rangle = - q(\lambda) \, \phi * k_t(x).\]
Hence we get
\[\phi * k_t = e^{-t q(\lambda)} \phi,\]
since $k_t$ is an approximate identity for $t \to 0^+$ (see Corollary~\ref{cor:kernelapproximateidentity}). This gives the former identity of \eqref{eq:eigenconvolution} when $m$ is a generator of $\JJ_L$, and consequently also for an arbitrary $m \in \JJ_L$; the latter identity follows by evaluating the former in $e$.
\end{proof}

The previous proposition shows that the joint eigenfunctions of $L_1,\dots,L_n$ are smooth, and are also eigenfunctions of the convolution operators with kernels in $\Kern_L(\JJ_L$). An analogous result holds in every unitary representation of $G$.

\begin{lem}\label{lem:eigenvectors}
Let $\pi$ be a unitary representation of $G$ on $\HH$. The following are equivalent for $v \in \HH \setminus \{0\}$:
\begin{itemize}
\item[(i)] $v \in \HH^\infty$ and $v$ is a joint eigenvector of $d\pi(L_1),\dots,d\pi(L_n)$;
\item[(ii)] $v$ is a joint eigenvector of the operators $\pi(\breve m)$ for $m \in \JJ_L$.
\end{itemize}
\end{lem}
\begin{proof}
(i) $\Rightarrow$ (ii) follows immediately from Proposition~\ref{prp:Jl1} and the properties of the spectral integral. For the reverse implication, take $m = e^{-p_j}$ for $j=0,\dots,n$, so that $\pi(\breve m) = e^{-p_j(d\pi(L))}$ by Proposition~\ref{prp:Jl1}; by the properties of the spectral integral, $\ker \pi(\breve m) = \{0\}$, therefore $\pi(\breve m) v = c v$ for some $c > 0$. This implies that
\[v = c^{-1} \pi(\breve m) v \in \HH^\infty,\]
by Theorem~\ref{thm:robinsonterelst}(b), and moreover, again by the properties of the spectral integral,
\[p_j(d\pi(L)) v = (\log c) v,\]
that is, $v$ is an eigenvector of $p_j(d\pi(L))$ for $j=0,\dots,n$. Since
\[\lambda_j = p_j(\lambda) - p_0(\lambda) \qquad\text{for $j=1,\dots,n$,}\]
it follows that $v$ is a joint eigenvector of $d\pi(L_1),\dots,d\pi(L_n)$.
\end{proof}

The link between eigenfunctions on $G$ and eigenvectors in unitary representations is given by the \emph{joint eigenfunctions of positive type}. Recall that a function of positive type $\phi : G \to \C$ is a diagonal coefficient for some unitary representation $\pi$ of $G$ on a Hilbert space $\HH$, i.e.,
\begin{equation}\label{eq:positivetype}
\phi(x) = \langle \pi(x) v, v \rangle
\end{equation}
for some vector $v \in \HH$, which can be supposed to be cyclic for $\pi$; in that case, the representation $\pi$ is uniquely determined by $\phi$ up to equivalence (see \S3.3 of \cite{folland_course_1995} for details), and is said to be \emph{associated} with $\phi$.

\begin{prp}\label{prp:eigenfunctions}
For a function of positive type $\phi$ on $G$, the following are equivalent:
\begin{itemize}
\item[(i)] $\phi$ is a joint eigenfunction of $L_1,\dots,L_n$ and $\phi(e) = 1$;
\item[(ii)] $\phi$ has the form \eqref{eq:positivetype} for some unitary representation $\pi$ of $G$ on $\HH$ and some cyclic vector $v$ of norm $1$, where $v \in \HH^\infty$ is a joint eigenvector of $d\pi(L_1),\dots,d\pi(L_n)$;
\item[(iii)] $\phi \neq 0$ and, for all $m \in \JJ_L$ and $f \in L^1(G)$,
\[\langle \breve m * f, \phi \rangle = \langle f * \breve m, \phi \rangle = \langle f, \phi \rangle \langle \breve m, \phi \rangle;\]
\item[(iv)] $\phi \neq 0$ and, for all $m \in \JJ_L$, $\langle \breve m * \breve m^*, \phi \rangle = |\langle \breve m, \phi \rangle|^2$.
\end{itemize}
In this case, moreover, the eigenvalue of $L_j$ corresponding to $\phi$ is a real number and coincides with the eigenvalue of $d\pi(L_j)$ corresponding to $v$.
\end{prp}
\begin{proof}
(i) $\Rightarrow$ (ii). Since $\phi$ is of positive type and $\phi(e) = 1$, then $\phi$ is of the form \eqref{eq:positivetype} for some unitary representation $\pi$ of $G$ on $\HH$ and some cyclic vector $v$ of norm $1$. From (i) we have $L_j \phi = \lambda_j \phi$ for some $\lambda = (\lambda_1,\dots,\lambda_n) \in \C^n$. Being $L_1,\dots,L_n$ left-invariant, if
\[\phi_y(x) = \LA_y \phi(x) = \langle \pi(x) v, \pi(y) v \rangle,\]
then also $L_j \phi_y = \lambda_j \phi_y$. Since $v$ is cyclic, for all $w \in \HH$ we can find a sequence $(w_n)_n$ in $\Span\{\pi(y) v \tc y \in G\}$ such that $w_n \to w$ in $\HH$; if
\[\psi_n(x) = \langle \pi(x) v, w_n \rangle, \qquad \psi(x) = \langle \pi(x) v, w \rangle,\]
then the $\psi_n$ are linear combinations of the $\phi_y$, so that $L_j \psi_n = \lambda_j \psi_n$ and, passing to the limit, we also have $L_j \psi = \lambda_j \psi$ in the sense of distributions. But then $\psi \in \E(G)$ by Proposition~\ref{prp:weakstrongeigenfunctions}. Since $w \in \HH$ was arbitrary, we conclude that $v \in \HH^\infty$; moreover
\[\langle \lambda_j v, w \rangle = \lambda_j \psi(e) = L_j \psi(e) = \langle d\pi(L_j) v, w \rangle,\]
and again, from the arbitrariness of $w$, we get $d\pi(L_j) v = \lambda_j v$ for $j=1,\dots,n$. Finally, since $d\pi(L_j)$ is self-adjoint, we deduce that $\lambda_j \in \R$.

(ii) $\Rightarrow$ (i). Trivial.

(ii) $\Rightarrow$ (iii). If $m \in \JJ_L$, by Lemma~\ref{lem:eigenvectors}, $\pi(\breve m)^* v = \pi(\breve{\overline{m}}) v = c v$ for some $c \in \C$. Since $\|v\| = 1$, we have
\begin{multline*}
\langle f * \breve m, \phi \rangle = \langle \pi(f * \breve m) v, v \rangle = \langle \pi(\breve m) \pi(f) v, v \rangle = \overline{c} \langle \pi(f) v,v \rangle \\
= \langle \pi(f) v,v \rangle \langle \pi(\breve m) v, v \rangle = \langle f, \phi \rangle \langle \breve m, \phi \rangle.
\end{multline*}
The other identity is proved analogously.

(iii) $\Rightarrow$ (iv). Trivial.

(iv) $\Rightarrow$ (ii). Being of positive type, $\phi$ has the form \eqref{eq:positivetype} for some unitary representation $\pi$ of $G$ on $\HH$ and some cyclic vector $v$. Then (iv) can be equivalently rewritten as
\begin{equation}\label{eq:multiplicativenorm}
\| \pi(\breve m) v \| = |\langle \pi(\breve m) v, v \rangle|
\end{equation}
for all $m \in \JJ_L$. In particular, by taking $m = e^{-tp_*}$, which is an approximate identity for $t \to 0^+$ (see Corollary~\ref{cor:kernelapproximateidentity}), and passing to the limit, we obtain $\|v\| = \|v\|^2$, so that $\|v\| = 1$ (since $\phi \neq 0$). Now, for an arbitrary $m \in \JJ_L$, \eqref{eq:multiplicativenorm} implies that $\pi(\breve m) v$ cannot have a component orthogonal to $v$, thus $v$ is an eigenvector of $\pi(\breve m)$, and (ii) follows from Lemma~\ref{lem:eigenvectors}.
\end{proof}

Let $\PP_L$ be the set of the joint eigenfunctions $\phi$ of $L_1,\dots,L_n$ of positive type with $\phi(e) = 1$. For every $\phi \in \PP_L$, by Proposition~\ref{prp:eigenfunctions} the corresponding eigenvalue $\lambda$ is in $\R^n$; we then define $\evmap_L : \PP_L \to \R^n$ by setting $\evmap_L(\phi) = \lambda$.

\begin{lem}\label{lem:continuouseigenvalues}
If $\PP_L$ is endowed with the topology induced by the weak-$*$ topology of $L^\infty(G)$, then the map $\evmap_L : \PP_L \to \R^n$ is continuous.
\end{lem}
\begin{proof}
By Proposition~\ref{prp:weakstrongeigenfunctions}, for $j=0,\dots,n$, we have that
\[e^{-p_j(\evmap_L(\phi))} = \langle \Kern_L (e^{-p_j}), \phi \rangle,\]
which is continuous in $\phi$ with respect to the weak-$*$ topology of $L^\infty(G)$. In particular, if $\evmap_{L,j} : \PP_L \to \R$ is the $j$-th component of $\evmap_L$ for $j=1,\dots,n$, then
\[e^{-\evmap_{L,j}(\phi)} = e^{-p_j(\evmap_L(\phi))}/e^{-p_0(\evmap_L(\phi))};\]
therefore the components of $\evmap_L$ are continuous $\PP_L \to \R$.
\end{proof}

\begin{prp}\label{prp:eigentopologies}
The topologies on $\PP_L$ induced by the weak-$*$ topology of $L^\infty(G)$, the compact-open topology of $C(G)$ and the topology of $\E(G)$ coincide. Moreover, the map $\evmap_L : \PP_L \to \R^n$ is a continuous, proper and closed map. In particular, the image $\evmap_L(\PP_L)$ is a closed subset of $\R^n$ and its topology as a subspace of $\R^n$ coincides with the quotient topology induced by $\evmap_L$.
\end{prp}
\begin{proof}
Since $G$ is second-countable, the three aforementioned topologies on $\PP_L$ are all metrizable (cf.\ \cite{megginson_introduction_1998}, Corollary~2.6.20). In particular, in order to prove that they coincide, it is sufficient to show that they induce the same notion of convergence of sequences.

Let $(\phi_k)_k$ be a sequence in $\PP_L$. If $(\phi_k)_k$ converges in $\E(G)$, then a fortiori it converges in $C(G)$. Moreover, since $\|\phi_k\|_\infty = 1$ for all $k$, convergence in $C(G)$ implies weak-$*$ convergence in $L^\infty(G)$ by dominated convergence.

Suppose now that $\phi_k \to \phi \in \PP_L$ with respect to the weak-$*$ topology of $L^\infty(G)$. Take $m = e^{-p_*} \in \JJ_L$, so that $m > 0$. By Proposition~\ref{prp:weakstrongeigenfunctions}, for all $D \in \Diff(G)$, we then have
\[D\phi_k = \frac{\phi_k * D \breve m}{m(\evmap_L(\phi_k))}, \qquad D\phi = \frac{\phi * D \breve m}{m(\evmap_L(\phi))};\]
in particular, for every $x \in G$, since $\RA_x D \breve m \in L^1(G)$,
\[D\phi_k(x) = \frac{\langle \RA_x D \breve m, \phi_k \rangle}{m(\evmap_L(\phi_k))} \to \frac{\langle \RA_x D \breve m, \phi \rangle}{m(\evmap_L(\phi))} = D\phi(x)\]
by Lemma~\ref{lem:continuouseigenvalues}. Moreover, again by Lemma~\ref{lem:continuouseigenvalues}, $m(\evmap_L(\phi_k)) \geq c > 0$ for some $c$ and all $k$, so that $\|D\phi_k\|_\infty \leq c^{-1} \|D\breve m\|_1$. This means that, for all $D \in \Diff(G)$, the family $\{D\phi_k\}_k$ is equibounded; but then also, for all $D \in \Diff(G)$, the family $\{D\phi_k\}_k$ is equicontinuous, so that the previously proved pointwise convergence $D\phi_k \to D\phi$ is in fact uniform on compacta. By arbitrariness of $D \in \Diff(G)$, we have then proved that $\phi_k \to \phi$ in $\E(G)$.

Let now $K \subseteq \R^n$ be compact, and take a sequence $(\phi_k)_k$ in $\PP_L$ such that $\evmap_L(\phi_k) \in K$ for all $k$. As before, the sequence $(\phi_k)_k$ is equibounded and equicontinuous, so that, by the Ascoli-Arzel\`a theorem (see \cite{bourbaki_topology2}, \S X.2.5), we can find a subsequence $\phi_{k_h}$ which converges uniformly on compacta to a function $\phi \in C(G)$, and such that moreover $\evmap_L(\phi_{k_h})$ converges to some $\lambda \in K$. It is now easy to show that $\phi$ is of positive type and $\phi(e) = 1$; moreover, for all $\eta \in \D(G)$,
\[\langle L_j \phi, \eta \rangle = \lim_h \langle L_j \phi_{k_h}, \eta \rangle = \lim_h \evmap_{L,j}(\phi_{k_h}) \langle \phi_{k_h}, \eta \rangle = \lambda_j \langle \phi, \eta \rangle,\]
so that, by Proposition~\ref{prp:weakstrongeigenfunctions}, $\phi$ is a (smooth) joint eigenfunction of $L_1,\dots,L_n$, hence $\phi \in \PP_L$. Since $\PP_L$ is metrizable, this shows that $\evmap_L^{-1}(K)$ is compact in $\PP_L$. By the arbitrariness of the compact $K \subseteq \R^n$, we conclude that $\evmap_L$ is proper and closed (see \cite{bourbaki_topology1}, Propositions I.10.1 and I.10.7).
\end{proof}

The following result, together with the Krein-Milman theorem, shows that the image of $\evmap_L$ does not change if we restrict to the joint eigenfunctions associated with \emph{irreducible} representations.

\begin{prp}\label{prp:eigenextremals}
For $\lambda \in \R^n$, the set $\evmap_L^{-1}(\lambda)$ is a weakly-$*$ compact and convex subset of $L^\infty(G)$, whose extreme points are the ones associated with irreducible representations.
\end{prp}
\begin{proof}
Clearly $\evmap_L^{-1}(\lambda)$ is convex, whereas compactness follows from Proposition~\ref{prp:eigentopologies}. In order to conclude, it will be sufficient to show that the extreme points of $\evmap_L^{-1}(\lambda)$ are also extreme points of the set $\PP_1$ of the functions $\phi$ of positive type on $G$ such that $\phi(e) = 1$ (see \cite{folland_course_1995}, Theorem~3.25). 

Suppose then that $\phi \in \evmap_L^{-1}(\lambda)$ is not extreme in $\PP_1$, i.e.,
\[\phi = \theta_0^2 \phi_0 + \theta_1^2 \phi_1\]
for some $\phi_0,\phi_1 \in \PP_1$ different from $\phi$ and some $\theta_0,\theta_1 > 0$ with $\theta_0^2 + \theta_1^2 = 1$. For $k=0,1$, we have $\phi_k(x) = \langle \pi_k(x) v_k, v_k \rangle$, where $\pi_k$ is a unitary representation of $G$ on $\HH_k$ and $v_k$ is a cyclic vector of norm $1$. If
\[v = (\theta_0 v_0, \theta_1 v_1) \in \HH_0 \oplus \HH_1, \qquad \HH = \overline{\Span\{(\pi_0 \oplus \pi_1)(x) v \tc x \in G\}},\]
and $\pi$ is the restriction of $\pi_0 \oplus \pi_1$ to $\HH$, then it is easy to see that $v$ is a cyclic vector for $\pi$ and that $\phi(x) = \langle \pi(x) v, v \rangle,$ therefore by Proposition~\ref{prp:eigenfunctions} it follows that $v \in \HH^\infty$ and that $d\pi(L_j) v = \lambda_j v$ for $j=1,\dots,n$.

If $P_k : \HH \to \HH_k$ is the restriction of the canonical projection $\HH_0 \oplus \HH_1 \to \HH_k$, it is immediate to check that $P_k$ intertwines $\pi$ and $\pi_k$, and that $P_k v = \theta_k v_k$; hence, for all $w \in \HH_k$ and $x \in G$,
\[\langle \pi_k(x) v_k, w \rangle = \theta_k^{-1} \langle \pi_k(x) P_k v , w \rangle = \theta_k^{-1} \langle \pi(x) v, P_k^* w \rangle.\]
This identity, together with the arbitrariness of $w \in \HH_k$, shows that also $v_k \in \HH_k^\infty$. Moreover, since $P_k$ intertwines $\pi(x)$ and $\pi_k(x)$ for all $x \in G$, it is easy to check that it intertwines also $d\pi(D)$ and $d\pi_k(D)$ for all $D \in \Diff(G)$, therefore
\[d\pi(L_j) v_k = \theta_k^{-1} P_k d\pi(L_j) v = \lambda_j v_k\]
for $j=1,\dots,n$. By Proposition~\ref{prp:eigenfunctions}, this shows that $\phi_0,\phi_1 \in \evmap_L^{-1}(\lambda)$, thus $\phi$ is not even extreme in $\evmap_L^{-1}(\lambda)$.
\end{proof}

In order to relate the joint spectrum of $L_1,\dots,L_n$ with (some subset of) $\evmap_L(\PP_L)$, we recall the notion of \emph{weak containment} of representations. If $\pi,\varpi$ are unitary representations of $G$, then $\pi$ is said to be \emph{weakly contained} in $\varpi$ if 
\[\|\pi(f)\| \leq \|\varpi(f)\| \qquad\text{for all $f \in L^1(G)$.}\]
Equivalent characterizations of weak containment can be given involving functions of positive type (cf.\ also \S3.5 of \cite{greenleaf_invariant_1969} and \S3.4 of \cite{dixmier_algebras_1982}):

\begin{lem}\label{lem:weakcontainment}
Let $\varpi$ be a unitary representation of $G$. Let moreover $\phi$ be a function of positive type, of the form \eqref{eq:positivetype} for some unitary representation $\pi$ of $G$ on the Hilbert space $\HH$ and some cyclic vector $v$ of unit norm. Then the following are equivalent:
\begin{itemize}
\item[(i)] $\pi$ is weakly contained in $\varpi$;
\item[(ii)] $|\langle f, \phi \rangle| \leq \|\varpi(f)\|$ for all $f \in L^1(G)$;
\item[(iii)] $|\langle f, \phi \rangle| \leq C \|\varpi(f)\|$ for some $C > 0$ and all $f \in L^1(G)$.
\end{itemize}
\end{lem}
\begin{proof}
(i) $\Rightarrow$ (ii) $\Rightarrow$ (iii). Trivial.

(iii) $\Rightarrow$ (i). Let $\widetilde\HH$ be the Hilbert space on which $\varpi$ acts. The hypothesis (iii) implies that $\phi$ defines a (positive) continuous functional on the sub-C$^*$-algebra of $\Bdd(\widetilde\HH)$ which is the closure of $\varpi(L^1(G))$. By applying Proposition~2.1.5(ii) of \cite{dixmier_algebras_1982} to this functional, one obtains, for $f,g \in L^1(G)$,
\[\|\pi(f) \pi(g) v\|^2 = \langle g * f* f^* * g^* , \phi \rangle \leq \|\varpi(f*f^*)\| \langle g * g^*, \phi \rangle = \|\varpi(f)\|^2 \|\pi(g)v\|^2.\]
Since $v$ is cyclic and $L^1(G)$ contains an approximate identity, the set
\[\{\pi(g) v \tc g \in L^1(G)\}\]
is a dense subspace of $\HH$, therefore the previously proved inequality gives (i).
\end{proof}

For a unitary representation $\varpi$ of $G$, we denote by $\PP_{L,\varpi}$ the set of the functions $\phi \in \PP_L$ which satisfy the equivalent conditions of Lemma~\ref{lem:weakcontainment}.

\begin{prp}
Let $\varpi$ be a unitary representation of $G$. Then $\PP_{L,\varpi}$ is a closed subset of $\PP_L$. Moreover, for every $\lambda \in \R^n$, $\PP_{L,\varpi} \cap \evmap_L^{-1}(\lambda)$ is compact and convex, and its extreme points are the ones associated with irreducible representations.
\end{prp}
\begin{proof}
Condition (ii) of Lemma~\ref{lem:weakcontainment} is a convex and closed condition (with respect to the weak-$*$ topology of $L^\infty(G)$) for every $f \in L^1(G)$. Therefore $\PP_{L,\varpi}$ is closed in $\PP_L$, and moreover, for $\lambda \in \R^n$, since $\evmap_L^{-1}(\lambda)$ is compact and convex (see Proposition~\ref{prp:eigenextremals}), $\PP_{L,\varpi} \cap \evmap_L^{-1}(\lambda)$ is compact and convex too.

In order to conclude, again by Proposition~\ref{prp:eigenextremals}, it is sufficient to show that an extreme point $\phi$ of $\PP_{L,\varpi} \cap \evmap_L^{-1}(\lambda)$ is also extreme in $\evmap_L^{-1}(\lambda)$. Suppose then that $\phi = (1-\theta)\phi_0 + \theta \phi_1$ for some $\phi_0,\phi_1 \in \evmap_L^{-1}(\lambda)$ and $0 < \theta < 1$. For $f \in L^1(G)$, we have
\[(1-\theta) |\langle f, \phi_0\rangle|^2 + \theta |\langle f, \phi_1\rangle|^2 = \langle f * f^* , \phi \rangle \leq \|\varpi(f)\|^2\]
by Lemma~\ref{lem:weakcontainment} and positivity, therefore
\[|\langle f, \phi_0 \rangle| \leq (1-\theta)^{-1/2} \|\varpi(f)\|, \qquad |\langle f, \phi_1 \rangle| \leq \theta^{-1/2} \|\varpi(f)\|,\]
and again by Lemma~\ref{lem:weakcontainment} we obtain $\phi_0,\phi_1 \in \PP_{L,\varpi} \cap \evmap_L^{-1}(\lambda)$.
\end{proof}

\begin{thm}\label{thm:spectrumeigenfunctions}
Let $\varpi$ be a unitary representation of $G$ on a Hilbert space $\HH$. Then $\evmap_L(\PP_{L,\varpi})$ is the joint spectrum of $d\varpi(L_1),\dots,d\varpi(L_n)$ on $\HH$.
\end{thm}
\begin{proof}
Let $E_\varpi$ be the joint spectral resolution of $d\varpi(L_1),\dots,d\varpi(L_n)$. The joint spectrum of $d\varpi(L_1),\dots,d\varpi(L_n)$, i.e., the support of $E_\varpi$, can be identified with the Gelfand spectrum of the C$^*$-algebra $E_\varpi[C_0(\R^n)]$ (cf.\ the proof of Proposition~\ref{prp:Jdensity}), i.e., with the closure in $\Bdd(\HH)$ of $\{\varpi(\breve m) \tc m \in \JJ_L\}$ (see Lemma~\ref{lem:Jdensity} and Proposition~\ref{prp:Jl1}).

In particular, if $\phi \in \PP_{L,\varpi}$, then, by Lemma~\ref{lem:weakcontainment},
\[|\langle \breve m, \phi \rangle| \leq \|\varpi(\breve m)\| \qquad\text{for all $m \in \JJ_L$,}\]
therefore $\phi$ defines a continuous functional on the C$^*$-algebra $E_\varpi[C_0(\R^n)]$, which is multiplicative by Proposition~\ref{prp:eigenfunctions}, and thus belongs to the Gelfand spectrum of $E_\varpi[C_0(\R^n)]$. Since
\[\langle \breve m, \phi \rangle = m(\evmap_L(\phi)) \qquad\text{for all $m \in \JJ_L$}\]
(see Proposition~\ref{prp:weakstrongeigenfunctions}), the element of $\supp E_\varpi$ corresponding to this functional is $\evmap_L(\phi)$.

Conversely, if $\lambda \in \supp E_\varpi$, then we can extend the corresponding character of $E_\varpi[C_0(\R^n)]$ to a positive functional $\omega$ of norm $1$ on the whole $\Bdd(\HH)$ (see \cite{dixmier_algebras_1982}, \S2.10). Since $\omega \circ \varpi : L^1(G) \to \C$ is linear and continuous, there exists $\phi \in L^\infty(G)$ such that
\[\langle f, \phi \rangle = \omega(\varpi(f)) \qquad\text{for all $f \in L^1(G)$;}\]
in fact, since $\omega$ is positive, $\phi$ must be a function of positive type on $G$ (see \cite{folland_course_1995}, \S3.3). Moreover, since $\omega$ extends a multiplicative functional on $E_\varpi[C_0(\R^n)]$, it must be
\[\langle \breve m_1 * \breve m_2, \phi \rangle = \langle \breve m_1, \phi \rangle \langle \breve m_2, \phi \rangle \qquad\text{for all $m_1,m_2 \in \JJ_L$.}\]
Therefore, by Proposition~\ref{prp:eigenfunctions}, $\phi \in \PP_L$, and in fact $\phi \in \PP_{L,\varpi}$ since $|\langle f, \phi\rangle| \leq \|\varpi(f)\|$ (see Lemma~\ref{lem:weakcontainment}). Finally
\[m(\evmap_L(\phi)) = \langle \breve m, \phi \rangle = \omega(\varpi(\breve m)) = m(\lambda) \qquad\text{for all $m \in \JJ_L$,}\]
by Proposition~\ref{prp:weakstrongeigenfunctions}, since $\omega$ extends the character corresponding to $\lambda$, and consequently $\evmap_L(\phi) = \lambda$ (see Lemma~\ref{lem:Jdensity}).
\end{proof}

In particular, the joint $L^2$ spectrum $\Sigma$ of $L_1,\dots,L_n$ coincides with the set of eigenvalues $\rho(\PP_{L,\RA})$ associated with the regular representation $\RA$ of $G$ on $L^2(G)$. When $G$ is amenable, every unitary representation is weakly contained in the regular representation (see \cite{greenleaf_invariant_1969}, \S3.5), hence

\begin{cor}\label{cor:spectrumeigenfunctions}
We have
\begin{equation}\label{eq:spectrainclusion}
\Sigma \subseteq \evmap_L(\PP_L),
\end{equation}
with equality when $G$ is amenable.
\end{cor}

Notice that, when $G$ is not amenable, the inclusion \eqref{eq:spectrainclusion} can be strict: for instance, if $n = 1$ and $L_1$ is a sublaplacian, then $0 \in \evmap_L(\PP_L) \setminus \Sigma$, since $L_1$ has a spectral gap (cf.\ \cite{varopoulos_hardy_1995}).

Under a more restrictive hypothesis than amenability, viz., the symmetry of the Banach $*$-algebra $L^1(G)$, we can relate the joint spectrum of $L_1,\dots,L_n$ to the Gelfand spectrum of a closed $*$-subalgebra of $L^1(G)$ (cf.\  \cite{hulanicki_spectrum_1972,hulanicki_subalgebra_1974,hulanicki_commutative_1975} for the case of a single operator). Namely, let $\II_L$ be the closure of $\Kern_L(\JJ_L)$ in $L^1(G)$. $\II_L$ is a commutative Banach $*$-subalgebra of $L^1(G)$, and also, by Proposition~\ref{prp:Jdensity}, a dense $*$-subalgebra of the $C^*$-algebra $C_0(L)$.

\begin{lem}\label{lem:hermitiancharacters}
Suppose that $L^1(G)$ is symmetric. Then every character of $\II_L$ extends to a character of $C_0(L)$, so that the Gelfand spectra of the two Banach $*$-algebras coincide (also as topological spaces).
\end{lem}
\begin{proof}
Since $G$ is connected and $L^1(G)$ is symmetric, then $G$ is also amenable (see \cite{palmer_banach_2001}, Theorem~12.5.18(e)), so that
\[\|f\|_{\Cv^2} = \sqrt{\rho(f^* f)} \qquad\text{for all $f \in L^1(G)$},\]
where $\rho(f)$ denotes the spectral radius of $f$ in $L^1(G)$ (see \cite{palmer_banach_2001}, Theorem~11.4.1, and also \cite{palmer_classes_1978}, p.\ 695). Notice that, since $\II_L$ is a closed subalgebra of $L^1(G)$, for every $f \in \II_L$, the spectral radius of $f$ in $\II_L$ coincides with its spectral radius in $L^1(G)$ (see \cite{bonsall_complete_1973}, Proposition~I.5.12). Moreover, since $L^1(G)$ is symmetric, also $\II_L$ is symmetric. Hence, for every character $\psi \in \GS(\II_L)$,
\[\psi(f^*) = \overline{\psi(f)} \qquad\text{for all $f \in \II_L$;}\]
since $\psi(f)$ belongs to the spectrum of $f$ for every $f \in \II_L$, we have
\[|\psi(f)|^2 = \psi(f^* f) \leq \rho(f^* f) = \|f\|_{\Cv^2}^2.\]
This shows that every character $\psi \in \GS(\II_L)$ is continuous with respect to the norm of $C_0(L)$, so that it extends by density to a unique character of $C_0(L)$.

Notice that, since $\II_L$ is dense in $C_0(L)$ and the elements of $\GS(C_0(L))$ have norms bounded by $1$ as functionals on $C_0(L)$, it is easy to check that the topologies of $\GS(C_0(L))$ and $\GS(\II_L)$ coincide.
\end{proof}

Finally we obtain that, if $L^1(G)$ is symmetric, then the joint $L^2$ spectrum of $L_1,\dots,L_n$ is the set of eigenvalues corresponding to all the \emph{bounded joint eigenfunctions}.

\begin{prp}
If $L^1(G)$ is symmetric, then the map
\[\Lambda : \PP_L \ni \phi \mapsto \langle \cdot, \phi \rangle \in \GS(\II_L)\]
is surjective. In particular, every multiplicative linear functional on $\II_L$ extends to a bounded linear functional $\eta$ on $L^1(G)$ such that
\begin{equation}\label{eq:semimultiplicative}
\eta(f * g) = \eta(f) \eta(g) \qquad\text{for all $f \in L^1(G)$ and $g \in \II_L$.}
\end{equation}
Moreover
\[\Sigma = \left\{ \lambda \in \C^n \tc L_j \phi = \lambda_j \phi \text{ for some $\phi \in L^\infty(G) \setminus \{0\}$ and all $j=1,\dots,n$} \right\}.\]
\end{prp}
\begin{proof}
Let $\psi \in \GS(\II_L)$. By Lemma~\ref{lem:hermitiancharacters}, $\psi$ extends to a character of $C_0(L)$, which corresponds to some $\lambda \in \Sigma$. Now, by Corollary~\ref{cor:spectrumeigenfunctions}, there exists $\phi \in \PP_L$ such that $\evmap_L(\phi) = \lambda$, therefore, for every $m \in \JJ_L$, by Proposition~\ref{prp:weakstrongeigenfunctions},
\[\Lambda(\phi)(\breve m) = \langle \breve m, \phi \rangle = m(\evmap_L(\phi)) = m(\lambda) = \psi(\breve m),\]
from which by density we deduce $\Lambda(\phi) = \psi$.

In particular, if $\eta$ denotes the linear functional $f \mapsto \langle f, \phi \rangle$ on $L^1(G)$, then $\eta$ extends $\psi$ and, by Proposition~\ref{prp:eigenfunctions},
\[\eta(f * \breve m) = \eta(f) \, \eta(\breve m) \qquad\text{for all $f \in L^1(G)$ and $m \in \JJ_L$,}\]
from which \eqref{eq:semimultiplicative} follows by density.

Finally, notice that every $\lambda \in \Sigma$ is, by Corollary~\ref{cor:spectrumeigenfunctions}, the eigenvalue corresponding to some $\phi \in \PP_L$, which is a bounded function. Vice versa, if $L_j \phi = \lambda_j \phi$ for some non-null $\phi \in L^\infty(G)$ and all $j=1,\dots,n$, then $\phi \in \E(G)$ by Proposition~\ref{prp:weakstrongeigenfunctions}; moreover, modulo replacing $\phi$ with $\LA_{x^{-1}} \phi / \phi(x)$ for some $x \in G$ with $\phi(x) \neq 0$, we may suppose that $\phi(e) = 1$. This means, again by Proposition~\ref{prp:weakstrongeigenfunctions}, that $\langle \cdot, \phi \rangle$ is a multiplicative linear functional on $\II_L$, hence $\langle \cdot, \phi \rangle = \langle \cdot, \psi \rangle$ on $\II_L$ for some $\psi \in \PP_L$, by surjectivity of $\Lambda$. Then necessarily $\lambda = \evmap_L(\psi) \in \Sigma$ by Proposition~\ref{prp:weakstrongeigenfunctions} and Corollary~\ref{cor:spectrumeigenfunctions}, since $G$ is amenable.
\end{proof}

\section{Examples}\label{section:examples}

\subsection{Homogeneous groups}\label{subsection:homogeneity}

Let $G$ be a homogeneous Lie group, with automorphic dilations $\delta_t$ and homogeneous dimension $Q_\delta$. A weighted subcoercive system $L_1,\dots,L_n$ on $G$ will be called \emph{homogeneous}\index{system of differential operators!weighted subcoercive!homogeneous} if each $L_j$ is $\delta_t$-homogeneous.

In the following, $L_1,\dots,L_n$ will be a homogeneous weighted subcoercive system, with associated Plancherel measure $\sigma$, and $r_j$ will denote the degree of homogeneity of $L_j$, i.e.,
\[\delta_t(L_j) = t^{r_j} L_j.\]
The unital subalgebra of $\Diff(G)$ generated by $L_1,\dots,L_n$ is $\delta_t$-invariant for every $t > 0$. Therefore, if we set
\[D_t f = f \circ \delta_{t^{-1}},\]
and if we denote by $\epsilon_t$ the dilations on $\R^n$ given by
\begin{equation}\label{eq:spectraldilations}
\epsilon_t(\lambda) = (t^{r_1} \lambda_1,\dots,t^{r_n} \lambda_n),
\end{equation}
then from Corollary~\ref{cor:automorphismskernel} we immediately deduce

\begin{prp}\label{prp:plancherelhomogeneous}
For every bounded Borel $m : \R^n \to \C$, we have
\[(m \circ \epsilon_t)(L) = D_t m(L) D_{t^{-1}}, \qquad (m \circ \epsilon_t)\breve{} = t^{-Q_\delta} \breve m \circ \delta_{t^{-1}}.\]
Moreover, the support $\Sigma$ of $\sigma$ is $\epsilon_t$-invariant, and
\[\sigma(\epsilon_t(A)) = t^{Q_\delta} \sigma(A)\]
for all Borel $A \subseteq \R^n$. In particular, the Plancherel measure $\sigma$ admits a ``polar decomposition'': if $S = \{\lambda \in \R^n \tc |\lambda|_\epsilon = 1\}$ for some $\epsilon_t$-homogeneous norm $|\cdot|_\epsilon$, then there exists a regular Borel measure $\tau$ on $S$ such that
\[\int_{\R^n} f \, d\sigma = \int_0^{+\infty} \int_S f(\epsilon_t(\omega)) \,d\tau(\omega) \,t^{Q_\delta - 1}\,dt.\]
\end{prp}

In the context of homogeneous groups, an equivalent characterization of homogeneous weighted subcoercive systems can be given, which is analogous to the definition of Rockland operator.

\begin{thm}
Let $L_1,\dots,L_n \in \Diff(G)$ be homogeneous, pairwise commuting and formally self-adjoint.
\begin{itemize}
\item[(i)] If $L_1,\dots,L_n$ is a weighted subcoercive system, then the algebra generated by $L_1,\dots,L_n$ contains a Rockland operator if and only if the degrees of homogeneity of $L_1,\dots,L_n$ have a common multiple.

\item[(ii)] $L_1,\dots,L_n$ is a weighted subcoercive system if and only if, for every non-trivial irreducible unitary representation $\pi$ of $G$ on a Hilbert space $\HH$, the operators $d\pi(L_1), \dots, d\pi(L_n)$ are \emph{jointly injective} on $\HH^\infty$, i.e.,
\[d\pi(L_1) v = \dots = d\pi(L_n) v = 0 \qquad\Longrightarrow\qquad v = 0\]
for all $v \in \HH^\infty$.
\end{itemize}
\end{thm}
\begin{proof}
Suppose that $L_1,\dots,L_n$ is a weighted subcoercive system. Let $p$ be a real polynomial such that $p(L) = p(L_1,\dots,L_n)$ is a weighted subcoercive operator. Choose moreover a system $X_1,\dots,X_d$ of generators of $\lie{g}$ made of $\delta_t$-homogeneous elements, so that $\delta_t(X_k) = t^{\nu_k} X_k$ for some $\nu_k > 0$. From Theorem~\ref{thm:robinsonterelst}(iii) we deduce that, possibly by replacing $p$ with some power $p^m$, there exist a constant $C > 0$ such that, for every unitary representation $\pi$ of $G$ on a Hilbert space $\HH$,
\begin{equation}\label{eq:apriori}
\|d\pi(X_k) v\|^2 \leq C ( \|v\|^2 + \|d\pi(p(L)) v\|^2 )
\end{equation}
for $v \in \HH^\infty$, $k=1,\dots,d$. Fix a non-trivial irreducible unitary representation $\pi$ of $G$ on a Hilbert space $\HH$, and let $v \in \HH^\infty$ be such that
\[d\pi(L_1) v = \dots = d\pi(L_n) v = 0.\]
For $t > 0$, since $\delta_t \in \Aut(G)$, $\pi_t = \pi \circ \delta_t$ is also a unitary representation of $G$; moreover, it is easily checked that smooth vectors for $\pi_t$ coincide with smooth vectors for $\pi$, and that
\[d\pi_t(D) = d\pi(\delta_t(D)) \qquad\text{for every $D \in \Diff(G)$.}\]
In particular,
\[d\pi_t(p(L)) v = d\pi((p \circ \epsilon_t)(L)) v = p(0) v,\]
thus from \eqref{eq:apriori} applied to the representation $\pi_t$ we get
\[\|d\pi(X_k) v\|^2 \leq t^{-2\nu_k} C (1 + |p(0)|^2) \|v\|^2,\]
and, for $t \to +\infty$, we obtain
\[d\pi(X_1) v = \dots = d\pi(X_d) v = 0.\]
Since $X_1,\dots,X_d$ generate $\lie{g}$, this means that the function $x \mapsto \pi(x) v$ is constant, i.e.,
\[\pi(x) v = v \qquad\text{for all $x \in G$,}\]
but $\pi$ is irreducible and non-trivial, thus $v = 0$.

Suppose now conversely that $d\pi(L_1),\dots,d\pi(L_n)$ are jointly injective on $\HH^\infty$ for every non-trivial irreducible representation $\pi$ on a Hilbert space $\HH$, and that moreover the degrees $r_1,\dots,r_n$ of homogeneity of $L_1,\dots,L_n$ have a common multiple $M$. Then
\[\Delta = L_1^{2M/r_1} + \dots + L_n^{2M/r_n}\]
is homogeneous of degree $2M$ and belongs to the subalgebra of $\Diff(G)$ generated by $L_1,\dots,L_n$. Moreover, for every irreducible unitary representation $\pi$ of $G$ on $\HH$, and for every $v \in \HH^\infty$, we have
\[\langle d\pi(\Delta) v, v \rangle = \|d\pi(L_1)^{M/r_1} v\|^2_\HH + \dots + \|d\pi(L_n)^{M/r_n} v\|^2_\HH,\]
so that, if $d\pi(\Delta) v = 0$, then also $d\pi(L_j) v = 0$ for $j=1,\dots,n$, therefore $v = 0$. This proves that $\Delta$ is a (positive) Rockland operator, and in particular it is weighted subcoercive, so that $L_1,\dots,L_n$ is a weighted subcoercive system.

If instead $d\pi(L_1),\dots,d\pi(L_n)$ are jointly injective for every non-trivial irreducible representation $\pi$, but the degrees of homogeneity of $L_1,\dots,L_n$ do not have a common multiple, by the results of \cite{miller_parametrices_1980} (see in particular Proposition~1.1 and its proof), we can find another homogeneous structure on $G$ with integral degrees, with respect to which the operators $L_1,\dots,L_n$ are still homogeneous. In particular the degrees of homogeneity of $L_1,\dots,L_n$ in this new structure must have a common multiple, so that, by the previous part of the proof, $L_1,\dots,L_n$ is again a weighted subcoercive system, and this last notion is independent of the homogeneous structure.

Finally, if the algebra generated by $L_1,\dots,L_n$ contains a Rockland operator, then (see \cite{miller_parametrices_1980}, Proposition~1.3; see also \cite{ter_elst_spectral_1997}) the homogeneity degrees of the elements of $\lie{g}$ must have a common multiple, and \emph{a fortiori} this is true also for the degrees of $L_1,\dots,L_n$.
\end{proof}

Notice that, while the existence of a Rockland operator on $G$ forces the homogeneity degrees of $\lie{g}$ to have a common multiple, this is not the case for the existence of a homogeneous weighted subcoercive system. For instance, the system of the partial derivatives $-i\partial_1,\dots,-i\partial_n$ on $\R^n$ is a homogeneous weighted subcoercive system with respect to any family of dilations of the form
\[\delta_t(x_1,\dots,x_n) = (t^{\lambda_1} x_1,\dots,t^{\lambda_n} x_n)\]
for $\lambda_1,\dots,\lambda_n \in \left[1,+\infty\right[$.

\subsection{Direct products}\label{section:directproducts}
In order to have a system of commuting operators, the simplest way is to start from operators living on different Lie groups, and then to consider them as operators on the direct product of the groups. Here we show that the notion of weighted subcoercive system is compatible with this construction, in the sense that weighted subcoercive systems on different groups can be put together in a single weighted subcoercive system on the direct product.

For $l = 1,\dots,\ell$, let $G_l$ be a connected Lie group, and set
\[G^\times = G_1 \times \dots \times G_\ell.\]
We then have the identification
\[\lie{g}^\times = \lie{g}_1 \oplus \dots \oplus \lie{g}_\ell.\]
Moreover, for $l=1,\dots,\ell$, if $D \in \Diff(G_l)$ and $D^\times$ is the image of $D$ via the derivative of the canonical inclusion $G_l \to G^\times$, then
\[D^\times (f_1 \otimes \cdots \otimes f_\ell) = f_1 \otimes \cdots \otimes f_{l-1} \otimes (D f_l) \otimes f_{l+1} \otimes \cdots \otimes f_\ell;\]
in this case, we say that $D^\times$ is the differential operator \emph{along the $l$-th factor} of $G^\times$ corresponding to $D \in \Diff(G_l)$.

\begin{lem}\label{lem:concatweighted}
For $l=1,\dots,\ell$, suppose that $A_{l,1},\dots,A_{l,d_l}$ is a reduced basis of $\lie{g}_l$, with weights $w_{l,1},\dots,w_{l,d_l}$. Then
\begin{equation}\label{eq:concatweighted}
A_{1,1},\dots,A_{1,d_1},\dots,A_{\ell,1},\dots,A_{\ell,d_\ell}
\end{equation}
is a reduced basis of $\lie{g}^\times$, with weights
\[w_{1,1},\dots,w_{1,d_1},\dots,w_{\ell,1},\dots,w_{\ell,d_\ell}.\]
Moreover, if $(V_{l,\lambda})_\lambda$ is the filtration on $\lie{g}_l$ corresponding to the chosen reduced basis for $l=1,\dots,\ell$, then
\[V^\times_\lambda = V_{1,\lambda} \oplus \dots \oplus V_{\ell,\lambda}\]
gives the filtration on $\lie{g}^\times$ corresponding to the algebraic basis \eqref{eq:concatweighted}; therefore, by passing to the quotients, we obtain for the contractions
\[(\lie{g}^\times)_* = (\lie{g}_1)_* \oplus \dots \oplus (\lie{g}_\ell)_*.\]
\end{lem}
\begin{proof}
An iterated commutator $A_{[\alpha]}$ of the elements of \eqref{eq:concatweighted} is not null only if it coincides with an iterated commutator $(A_l)_{[\alpha']}$ of $A_{l,1},\dots,A_{l,n_l}$ for some $l \in \{1,\dots,\ell\}$. This can be easily checked by induction on the length $|\alpha|$ of the commutator. The identities involving the filtrations then follow immediately, from which we get easily the conclusion.
\end{proof}

\begin{thm}\label{thm:productwsub}
Suppose that $D_l \in \Diff(G_l)$ is a self-adjoint weighted subcoercive operator on $G_l$, for $l = 1,\dots,\ell$, and let $D_l^\times \in \Diff(G^\times)$ be the differential operator on $G^\times$ along the $l$-th factor corresponding to $D_l$. Then
\[D = (D_1^\times)^2 + \dots + (D_\ell^\times)^2\]
is a positive weighted subcoercive operator on $G^\times$.
\end{thm}
\begin{proof}
For $l=1,\dots,\ell$, let $A_{l,1},\dots,A_{l,d_l}$ be a reduced basis of $\lie{g}_l$, such that, for some self-adjoint weighted subcoercive form $C_l$, we have $D_l = dR_{G_l}(C_l)$; let moreover $P_l$ be the principal part of $C_l$. Clearly, modulo rescaling the weights of the reduced bases, we may suppose that the forms $C_1,\dots,C_\ell$ have the same degree $m$.

By Lemma~\ref{lem:concatweighted}, the concatenation of the bases of $\lie{g}_1,\dots,\lie{g}_l$ gives a reduced basis \eqref{eq:concatweighted} of $\lie{g}^\times$. We can then consider, for $l=1,\dots,\ell$, the forms $C_l^\times$, $P_l^\times$ corresponding to $C_l$, $P_l$ but re-indexed on the basis \eqref{eq:concatweighted}. In particular, if
\[C = (C_1^\times)^2 + \dots + (C_\ell^\times)^2, \qquad P = (P_1^\times)^2 + \dots + (P_\ell^\times)^2,\]
then $P = P^+$ is the principal part of $C$, and moreover
\[d\RA_{G^\times}(C) = (d\RA_{G_1}(C_1)^\times)^2 + \dots + (d\RA_{G_\ell}(C_\ell)^\times)^2 = D.\]

On the other hand, again by Lemma~\ref{lem:concatweighted}, we have the identification
\[(G^\times)_* = (G_1)_* \times \dots \times (G_\ell)_*,\]
so that
\[d\RA_{(G^\times)_*}(P) = (d\RA_{(G_1)_*}(P_1)^\times)^2 + \dots + (d\RA_{(G_\ell)_*}(P_\ell)^\times)^2.\]
By Theorem~\ref{thm:robinsonterelst}, we have that $d\RA_{(G_l)_*}(P_l)$ is Rockland on $(G_l)_*$ for $l=1,\dots,\ell$; in order to conclude, it is sufficient to show that $dR_{(G^\times)_*}(P)$ is Rockland on $(G^\times)_*$.

If $\pi$ is a non-trivial irreducible unitary representation of $G^\times$ on a Hilbert space $\HH$, then (see \cite{folland_course_1995}, Theorem~7.25) we may suppose that $\pi = \pi_1 \otimes \dots \otimes \pi_\ell$, where $\pi_l$ is an irreducible unitary representation of $G_l$ on a Hilbert space $\HH_l$ for $l=1,\dots,\ell$, so that $\HH = \HH_1 \mathop{\hat\otimes} \cdots \mathop{\hat\otimes} \HH_\ell$ and at least one of $\pi_1,\dots,\pi_\ell$ is non-trivial. Let $(w_{l,\nu_l})_{\nu_l}$ be a complete orthonormal system for $\HH_l$, for $l=1,\dots,\ell$, so that $(w_{1,\nu_1} \otimes \dots \otimes w_{\ell,\nu_{\ell}})_{\vec{\nu}}$ is a complete orthonormal system for $\HH$. Then, for every element $v = \sum_{\nu_1,\dots,\nu_\ell} a_{\nu_1,\dots,\nu_\ell} w_{1,\nu_1} \otimes \dots \otimes w_{\ell,\nu_\ell}$ of $\HH$, we have
\begin{multline*}
\langle d\pi(d\RA_{(G^\times)_*}(P)) v, v \rangle_{\HH} \\
= \sum_{l=1}^\ell \sum_{\nu_1,\dots,\nu_{l-1},\nu_{l+1},\nu_\ell} \left\| d\pi_l(d\RA_{(G_l)_*}(P_l))\left( \sum_{\nu_l} a_{\nu_1,\dots,\nu_\ell} w_{l,\nu_l} \right)\right\|^2_{\HH_l};
\end{multline*}
since at least one of the $d\pi_l(d\RA_{(G_l)_*}(P_l))$ is injective (being $d\RA_{(G_l)_*}(P_l)$ Rockland and $\pi_l$ non-trivial), this formula gives easily that
\[v \neq 0 \qquad\Longrightarrow\qquad d\pi(d\RA_{(G^\times)_*}(P)) v \neq 0,\]
i.e., $d\pi(d\RA_{(G^\times)_*}(P))$ is injective.
\end{proof}

Theorems~\ref{thm:productwsub} and \ref{thm:plancherel}, together with the properties of the spectral integral, yield easily

\begin{cor}\label{prp:productfunctional}
For $l=1,\dots,\ell$, let $L_{l,1},\dots,L_{l,n_l} \in \Diff(G_l)$ be a weighted subcoercive system. Let moreover $L_{l,j}^\times$ be the differential operator on $G^\times$ along the $l$-th factor corresponding to $L_{l,j}$. Then
\begin{equation}\label{eq:concatsystem}
L_{1,1}^\times,\dots,L_{1,n_1}^\times,\dots,L_{\ell,1}^\times,\dots,L_{\ell,n_\ell}^\times
\end{equation}
is a weighted subcoercive system on $G^\times$. Further:
\begin{itemize}
\item[(a)] if $m_l$ is a bounded Borel function on $\R^{n_l}$ for $l=1,\dots,\ell$, then
\[\Kern_{L^\times} m = \Kern_{L_1} m_1 \otimes \dots \otimes \Kern_{L_\ell} m_\ell;\]
\item[(b)] if $\sigma_l$ is the Plancherel measure associated with the system $L_{l,1},\dots,L_{l,n_l}$ for $l=1,\dots,\ell$, and if moreover $\sigma^\times$ is the Plancherel measure associated with the system \eqref{eq:concatsystem}, then
\[\sigma^\times = \sigma_1 \times \dots \times \sigma_\ell.\]
\end{itemize}
\end{cor}

\subsection{Gelfand pairs}\label{subsection:gelfandpairs}

Let $G$ be a connected Lie group. In this paragraph, we describe a particular way of obtaining weighted subcoercive systems on $G$, which has been extensively studied in the literature.

Let $K$ be a compact subgroup of $\Aut(G)$. A function (or distribution) $f$ on $G$ is said to be \emph{$K$-invariant} if
\[T_k f = f \qquad\text{for all $k \in K$.}\]
We add a subscript $K$ to the symbol representing a particular space of functions or distributions in order to denote the corresponding subspace of $K$-invariant elements; for instance, $L^p_K(G)$ denotes the Banach space of $K$-invariant $L^p$ functions on $G$. Since
\[T_k (f * g) = (T_k f) * (T_k g), \qquad T_k(f^*) = (T_k f)^*,\]
it is immediately proved that $L^1_K(G)$ is a Banach $*$-subalgebra of $L^1(G)$. We also define the projection onto $K$-invariant elements:
\[P_K : f \mapsto \int_K T_k f \,dk,\]
where the integration is with respect to the Haar measure on $K$ with mass $1$. This projection satisfies
\[P_K(f * (P_K g)) = P_K( (P_K f) * g) = (P_K f) * (P_K g), \qquad P_K(f^*) = (P_K f)^*.\]

Among the left-invariant differential operators on $G$, we can consider those which are $K$-invariant, i.e., which commute with $T_k$ for all $k \in K$. The set $\Diff_K(G)$ of left-invariant $K$-invariant differential operators on $G$ is a $*$-subalgebra of $\Diff(G)$, which is finitely generated since $K$ is compact (cf.\ \cite{helgason_differential_1962}, Corollary~X.2.8 and Theorem~X.5.6). Moreover, $\Diff_K(G)$ contains an elliptic operator (e.g., the Laplace-Beltrami operator associated with a left-invariant $K$-invariant metric on $G$, cf.\ \cite{helgason_groups_1984}, proof of Proposition~IV.2.2). Therefore, if one chooses a finite system of formally self-adjoint generators of $\Diff_K(G)$, the only property which is missing in order to have a weighted subcoercive system is commutativity of $\Diff_K(G)$.

In fact, under these hypotheses, the following properties are equivalent (cf.\ \cite{thomas_infinitesimal_1984}, or \cite{wolf_harmonic_2007}, \S8.3):
\begin{itemize}
\item $\Diff_K(G)$ is a commutative $*$-subalgebra of $\Diff(G)$;
\item $L^1_K(G)$ is a commutative Banach $*$-subalgebra of $L^1(G)$.
\end{itemize}
The latter condition corresponds to the fact that $(G \rtimes K, K)$ is a \emph{Gelfand pair}\index{Gelfand pair}\footnote{If $S$ is a locally compact group, and $K$ a compact subgroup of $S$, then $(S,K)$ is said to be a Gelfand pair if the (convolution) algebra $L^1(K;S;K)$ of bi-$K$-invariant integrable functions on $S$ is commutative. The study of a Gelfand pair $(S,K)$ involves the $K$-homogeneous space $S/K$. In the case $S = G \rtimes K$, the space $S/K$ can be identified with $G$, and most of the notions and results about Gelfand pairs can be rephrased in terms of the algebraic structure of $G$ (see, e.g., \cite{carcano_commutativity_1987,benson_gelfand_1990}); this has to be kept in mind when comparing the results presented in the literature with the ones mentioned here. Notice that, according to Vinberg's reduction theorem (see \cite{vinberg_commutative_2001}), Gelfand pairs in ``semidirect-product form'' are one of the two structural constituents of general Gelfand pairs.}. We now summarize in our context some of the main notions and results from the general theory of Gelfand pairs, for which we refer mainly to \cite{faraut_analyse_1982,wolf_harmonic_2007,helgason_differential_1962,helgason_groups_1984}. In the following, we always suppose that $L^1_K(G)$ is commutative; consequently, $G$ must be unimodular (cf.\ \cite{helgason_groups_1984}, Theorem~IV.3.1).

The $K$-invariant joint eigenfunctions $\phi$ of the operators in $\Diff_K(G)$ with $\phi(e) = 1$ are called \emph{$K$-spherical functions}. The set $\GS_K$ of bounded $K$-spherical functions, with the topology induced by the weak-$*$ topology of $L^\infty(G)$, is identified with the Gelfand spectrum $\GS(L^1_K(G))$ of the commutative Banach $*$-algebra $L^1_K(G)$, via the correspondence which associates to a bounded $K$-spherical function $\phi$ the (multiplicative) linear functional $f \mapsto \langle f, \phi \rangle$ on $L^1_K(G)$. According to this identification, the \index{transform!Gelfand}Gelfand transform --- which is also called the \emph{$K$-spherical Fourier transform} --- of an element $f \in L^1_K(G)$ is the function
\[\Gelf_K f : \GS_K \ni \phi \mapsto \langle f, \phi \rangle \in \C.\]

Let $\PP_K$ denote the set of $K$-invariant functions $\phi$ of positive type on $G$ with $\phi(e) = 1$. Then $\PP_K$ is a closed and convex subset of $\PP_1$, whose extreme points are the elements of $\GS_K^+ = \GS_K \cap \PP_K$, i.e., the $K$-spherical functions of positive type; in particular, by the Krein-Milman theorem, the convex hull of $\GS_K^+$ is weakly-$*$ dense in $\PP_K$. By restricting $K$-spherical transforms to $\GS_K^+$, one obtains that
\[(\Gelf_K (f^*))|_{\GS_K^+} = \overline{(\Gelf_K f)|_{\GS_K^+}},\]
therefore the map $f \mapsto (\Gelf_K f)|_{\GS_K^+}$ is a $*$-homomorphism $L^1_K(G) \to C_0(\GS_K^+)$ with unit norm and dense image. Moreover, there exists a unique positive regular Borel measure $\sigma_K$ on $\GS_K^+$, which is called the \index{Plancherel measure!for a Gelfand pair}\emph{Plancherel measure} of the Gelfand pair $(G \rtimes K, K)$, such that
\[\int_G |f(x)|^2 \,dx = \int_{\GS_K^+} |\Gelf_K f(\phi)|^2 \,d\sigma_K(\phi)\]
for all $f \in L^1_K \cap L^2_K(G)$; further, the map $f \mapsto (\Gelf_K f)|_{\GS_K^+}$ extends to an isomorphism $L^2_K(G) \to L^2(\GS_K^+,\sigma_K)$.

Choose now a finite system $L_1,
\dots,L_n$ of formally self-adjoint generators of $\Diff_K(G)$. As we have seen before, the system $L_1,\dots,L_n$ is a weighted subcoercive system on $G$. If the map $\evmap_L$ of \S\ref{section:eigenfunctions} is extended to all the joint eigenfunctions of $L_1,\dots,L_n$, then it is known (see \cite{ferrari_ruffino_topology_2007}) that
\[\evmap_L|_{\GS_K} : \GS_K \to \C^n\]
is a homeomorphism with its image $\evmap_L(\GS_K)$, which is a closed subset of $\C^n$. Notice that
\[\GS_K^+ \subseteq \PP_L, \qquad \evmap_L(\GS_K^+) = \evmap_L(\PP_L);\]
consequently, for every $\lambda \in \evmap_L(\PP_L)$, there exists a unique element of $\evmap_L^{-1}(\lambda) \cap \PP_L$ which is a $K$-spherical function (cf.\ \cite{helgason_groups_1984}, Proposition~IV.2.4).

The embedding $\evmap_L$ allows us to compare the notions of $K$-spherical transform $\Gelf_K$ and Plancherel measure $\sigma_K$ of the Gelfand pair $(G \rtimes K, K)$ with the notions of kernel transform $\Kern_L$ and Plancherel measure $\sigma$ associated with the weighted subcoercive system $L_1,\dots,L_n$. Notice that, in the case of nilpotent $G$ and Schwartz multipliers, results similar to the following are proved in \cite{astengo_gelfand_2009,fischer_gelfand_2008} (cf.\ also \S1.7 of \cite{gangolli_harmonic_1988}).

As a preliminary remark, notice that from Proposition~\ref{prp:automorphismskernel} it follows that, for every bounded Borel $m : \R^n \to \C$, the corresponding kernel $\Kern_L m$ is $K$-invariant.

\begin{prp}\label{prp:gelfandkernel}
Let $f \in L^1_K(G)$. Then there exists $m \in C_0(\R^n)$ such that
\[\Gelf_K f(\phi) = m(\evmap_L(\phi)) \qquad\text{for $\phi \in \GS_K^+$.}\]
For any of such $m$, and for every unitary representation $\pi$ of $G$, we have
\[\pi(f) = m(d\pi(L_1),\dots,d\pi(L_n)),\]
and in particular
\[f = \Kern_L m.\]
\end{prp}
\begin{proof}
Since $\Gelf_K f|_{\GS_K^+} \in C_0(\GS_K^+)$, and since $\evmap_L|_{\GS_K^+}$ is a homeomorphism with its image, which is a closed subset of $\R^n$, then by the Tietze-Urysohn extension theorem we can find $m \in C_0(\R^n)$ extending $(\Gelf_K f) \circ (\evmap_L|_{\GS_K^+})^{-1}$.

By Proposition~\ref{prp:Jl1}, for every $u \in \JJ_L$ and every unitary representation $\pi$ of $G$, we have
\[\pi(\breve u) = u(d\pi(L_1),\dots,d\pi(L_n));\]
therefore the map
\[\JJ_L \ni u \mapsto \breve u \in L^1(G)\]
extends by density (see Proposition~\ref{prp:Jdensity}) to a $*$-homomorphism
\[\Phi : C_0(\R^n) \to C^*(G),\]
and we have
\[\pi(\Phi(u)) = u(d\pi(L_1),\dots,d\pi(L_n))\]
for all $u \in C_0(\R^n)$ and all unitary representations $\pi$ of $G$. The conclusion will then follow if we prove that $f = \Phi(m)$ as elements of $C^*(G)$.

Recall that every $\phi \in \PP_1$ defines a positive continuous functional $\omega_\phi$ on $C^*(G)$ with unit norm, extending
\[L^1(G) \ni h \mapsto \langle h, \phi \rangle \in \C.\]
In fact, the norm of an arbitrary $g \in C^*(G)$ is given by
\[\|g\|_* = \sup_{\phi \in \PP_1} \omega_\phi(g * g^*)\]
(see \cite{folland_course_1995}, Proposition~7.1); therefore, in order to conclude, it will be sufficient to show that the set $A$ of the $\phi \in \PP_1$ such that
\[\omega_\phi( (f - \Phi(m)) * (f - \Phi(m))^* ) = 0\]
coincides with the whole $\PP_1$.

Notice that both $f$ and $\Phi(m)$ belong to the closure $C^*_K(G)$ of $L^1_K(G)$ in $C^*(G)$, and it is easily checked that, for $\phi \in \PP_1$ and $g \in C^*_K(G)$,
\[\omega_\phi(g) = \omega_{P_K \phi}(g);\]
consequently, we are reduced to prove that $\PP_K \subseteq A$. In fact, since $A$ is a closed convex subset of $\PP_1$, it is sufficient to prove the inclusion $\GS_K^+ \subseteq A$.

On the other hand, the functionals $\omega_\phi$ for $\phi \in \GS_K^+$ are multiplicative on $L^1_K(G)$, thus they are also multiplicative on $C^*_K(G)$ by continuity, therefore
\[\omega_\phi( (f - \Phi(m)) * (f - \Phi(m))^* ) = |\omega_\phi(f - \Phi(m))|^2 = |\Gelf_K f(\phi) - m(\evmap_L(\phi))|^2 = 0\]
for every $\phi \in \GS_K^+$, and we are done.
\end{proof}

Thus, by applying first $\Gelf_K$ and then $\Kern_L$, we are back at the beginning. The composition of the transforms in reverse order is considered in the following statement, which gives also an improvement of Proposition~\ref{prp:riemannlebesgue1} in this particular context.

\begin{cor}
Let $m : \R^n \to \C$ be a bounded Borel function such that $\breve m \in L^1(G)$. Then $\breve m \in L^1_K(G)$ and
\[\Gelf_K (\Kern_L m) (\phi) = m(\evmap_L(\phi)) \qquad\text{for all $\phi \in \GS_K^+$ with $\evmap_L(\phi) \in \Sigma$.}\]
In particular $m|_\Sigma \in C_0(\Sigma)$.
\end{cor}
\begin{proof}
We already know that $\breve m$ is $K$-invariant, so that $\breve m \in L^1_K(G)$. Therefore, by Proposition~\ref{prp:gelfandkernel}, we can find $u \in C_0(\R^n)$ such that
\[\Gelf_K \breve m (\phi) = u(\evmap_L(\phi))\]
for all $\phi \in \GS_K^+$, and we have $\breve m = \breve u$, i.e.,
\[m(L_1,\dots,L_n) = u(L_1,\dots,L_n),\]
which means that $m$ and $u$ must coincide on the joint spectrum $\Sigma$ of $L_1,\dots,L_n$, and we are done.
\end{proof}

Finally, we compare the Plancherel measures $\sigma$ and $\sigma_K$.

\begin{cor}
We have
\[\sigma = \evmap_L|_{\GS_K^+}(\sigma_K), \qquad \sigma_K = (\evmap_L|_{\GS_K^+})^{-1}(\sigma).\]
\end{cor}
\begin{proof}
Recall that $\evmap_L|_{\GS_K^+}$ is a homeomorphism with its image, which is a closed subset of $\R^n$ containing the support $\Sigma$ of $\sigma$, thus the two equalities to be proved are equivalent. 

Set $\tilde\sigma = (\evmap_L|_{\GS_K^+})^{-1}(\sigma)$. Then $\tilde\sigma$ is a positive regular Borel measure on $\GS_K^+$. Moreover, if $f \in L^1_K \cap L^2_K(G)$, then by Proposition~\ref{prp:gelfandkernel} there is $m \in C_0(\R^n)$ such that
\[\Gelf_K f(\phi) = m(\evmap_L(\phi)) \qquad\text{for all $\phi \in \GS_K^+$}\]
and
\[f = \breve m.\]
Since $f \in L^2(G)$, by Theorem~\ref{thm:plancherel} we also have $m \in L^2(\sigma)$, and
\[\int_G |f(x)|^2 \,dx = \int_{\R^n} |m|^2 \,d\sigma = \int_{\GS_K^+} |\Gelf_K f|^2 \,d\tilde\sigma\]
by the change-of-variable formula for push-forward measures. By the arbitrariness of $f \in L^1_K \cap L^2_K(G)$ and the uniqueness of the Plancherel measure of a Gelfand pair, we obtain that $\sigma_K = \tilde \sigma$, and we are done.
\end{proof}

We have thus shown that the study of the algebra $\Diff_K(G)$ of differential operators associated with a Gelfand pair $(G \rtimes K, K)$ fits into the more general setting of weighted subcoercive systems, where in general there is no compact group $K$ of automorphisms which determines the algebra of operators.

It should be noticed that the hypothesis of Gelfand pair is quite restrictive. We have already mentioned that, if $L^1_K(G)$ is commutative, then $G$ must be unimodular. Moreover, the algebra $\Diff_K(G)$ always contains an elliptic operator, while a general weighted subcoercive operator is not even analytic hypoelliptic (see, e.g., \cite{helffer_conditions_1982}). Further, if $G$ is solvable, then $G$ must have polynomial growth, and, if $G$ is nilpotent, then $G$ is at most $2$-step (see \cite{benson_gelfand_1990}).

In this last case, notice that it is always possible to find a family of automorphic dilations on $G$ which commute with the elements of $K$, and any system $L_1,\dots,L_n$ of homogeneous formally self-adjoint generators of $\Diff_K(G)$ is a homogeneous weighted subcoercive system. On the other hand, the results of this paper can be applied to homogeneous groups which are $3$-step or more, and which therefore do not belong to the realm of Gelfand pairs. Take for instance the free $3$-step nilpotent group $N_{2,3}$ with $2$ generators, defined by the relations
\[[X_1,X_2] = Y, \qquad [X_1,Y] = T_1, \qquad [X_2,Y] = T_2,\]
where $X_1,X_2,Y,T_1,T_2$ is a basis of its Lie algebra, and notice that the group $SO_2$ acts on $N_{2,3}$ by automorphisms given by simultaneous rotations of $\R X_1 + \R X_2$ and $\R T_1 + \R T_2$. Although the whole algebra of $SO_2$-invariant left-invariant differential operators on $N_{2,3}$ cannot be commutative, the operators
\[-(X_1^2 + X_2^2), \qquad 2 X_2 T_1 - 2 X_1 T_2 - Y^2, \qquad -(T_1^2 + T_2^2)\]
generate a non-trivial homogeneous commutative subalgebra to which our results apply, as well as they apply to the larger algebra generated by
\[-(X_1^2 + X_2^2), \qquad 2 X_2 T_1 - 2 X_1 T_2 - Y^2, \qquad -i T_1, \qquad -i T_2\]
(which is no longer made of $SO_2$-invariant operators).

\section*{Acknowledgements}
I thank Fulvio Ricci for introducing me to the theory of Gelfand pairs, and for his continuous support and encouragement. I also thank A.~F.~M.\ ter Elst for the exchange of comments about his work.

\bibliographystyle{abbrv}
\bibliography{../../multipliers}

\def\cprime{$'$} \def\dbar{\leavevmode\hbox to 0pt{\hskip.2ex \accent"16\hss}d}
\begin{thebibliography}{10}

\bibitem{astengo_gelfand_2009}
F.~Astengo, B.~Di~Blasio, and F.~Ricci.
\newblock Gelfand pairs on the {H}eisenberg group and {S}chwartz functions.
\newblock {\em J. Funct. Anal.}, 256(5):1565--1587, 2009.

\bibitem{auscher_positive_1994}
P.~Auscher, A.~F.~M. ter Elst, and D.~W. Robinson.
\newblock On positive {R}ockland operators.
\newblock {\em Colloq. Math.}, 67(2):197--216, 1994.

\bibitem{benson_gelfand_1990}
C.~Benson, J.~Jenkins, and G.~Ratcliff.
\newblock On {G}el\cprime fand pairs associated with solvable {L}ie groups.
\newblock {\em Trans. Amer. Math. Soc.}, 321(1):85--116, 1990.

\bibitem{berberian_notes_1966}
S.~K. Berberian.
\newblock {\em Notes on spectral theory}.
\newblock Van Nostrand Mathematical Studies, No. 5. D. Van Nostrand Co., Inc.,
  Princeton, N.J.-Toronto, Ont.-London, 1966.
\newblock Corrected second edition (2009) available on the web at
  \texttt{www.ma.utexas.edu/mp\_arc}.

\bibitem{berezanskii_expansions_1968}
J.~M. Berezans{\cprime}ki{\u\i}.
\newblock {\em Expansions in eigenfunctions of selfadjoint operators}.
\newblock Translated from the Russian by R. Bolstein, J. M. Danskin, J. Rovnyak
  and L. Shulman. Translations of Mathematical Monographs, Vol. 17. American
  Mathematical Society, Providence, R.I., 1968.

\bibitem{bonsall_complete_1973}
F.~F. Bonsall and J.~Duncan.
\newblock {\em Complete normed algebras}.
\newblock Springer-Verlag, New York, 1973.
\newblock Ergebnisse der Mathematik und ihrer Grenzgebiete, Band 80.

\bibitem{bourbaki_lie1}
N.~Bourbaki.
\newblock {\em Lie groups and {L}ie algebras. {C}hapters 1--3}.
\newblock Elements of Mathematics (Berlin). Springer-Verlag, Berlin, 1989.
\newblock Translated from the French, Reprint of the 1975 edition.

\bibitem{bourbaki_topology1}
N.~Bourbaki.
\newblock {\em General topology. {C}hapters 1--4}.
\newblock Elements of Mathematics. Springer-Verlag, Berlin, 1998.
\newblock Translated from the French, Reprint of the 1989 English translation.

\bibitem{bourbaki_topology2}
N.~Bourbaki.
\newblock {\em General topology. {C}hapters 5--10}.
\newblock Elements of Mathematics. Springer-Verlag, Berlin, 1998.
\newblock Translated from the French, Reprint of the 1989 English translation.

\bibitem{carcano_commutativity_1987}
G.~Carcano.
\newblock A commutativity condition for algebras of invariant functions.
\newblock {\em Boll. Un. Mat. Ital. B (7)}, 1(4):1091--1105, 1987.

\bibitem{christ_multipliers_1991}
M.~Christ.
\newblock {$L^p$} bounds for spectral multipliers on nilpotent groups.
\newblock {\em Trans. Amer. Math. Soc.}, 328(1):73--81, 1991.

\bibitem{dixmier_algebras_1982}
J.~Dixmier.
\newblock {\em {$C\sp*$}-algebras}.
\newblock North-Holland Publishing Co., Amsterdam, 1982.
\newblock Revised edition. Translated from the French by Francis Jellett,
  North-Holland Mathematical Library, Vol. 15.

\bibitem{faraut_analyse_1982}
J.~Faraut.
\newblock {Analyse harmonique sur les paires de Guelfand et les espaces
  hyperboliques}.
\newblock In {\em Analyse harmonique}, Les Cours du C.I.M.P.A., pages 315--446.
  CIMPA/ICPAM, Nice, 1982.

\bibitem{fell_representations_1988}
J.~M.~G. Fell and R.~S. Doran.
\newblock {\em Representations of {$^*$}-algebras, locally compact groups, and
  {B}anach {$^*$}-algebraic bundles. {V}ol. 1}, volume 125 of {\em Pure and
  Applied Mathematics}.
\newblock Academic Press Inc., Boston, MA, 1988.
\newblock Basic representation theory of groups and algebras.

\bibitem{ferrari_ruffino_topology_2007}
F.~Ferrari~Ruffino.
\newblock The topology of the spectrum for {G}elfand pairs on {L}ie groups.
\newblock {\em Boll. Unione Mat. Ital. Sez. B Artic. Ric. Mat. (8)},
  10(3):569--579, 2007.

\bibitem{fischer_gelfand_2008}
V.~Fischer and F.~Ricci.
\newblock Gelfand transforms of {${\rm SO}(3)$}-invariant {S}chwartz functions
  on the free group {$N_{3,2}$}.
\newblock {\em Ann. Inst. Fourier (Grenoble)}, 59(6):2143--2168, 2009.

\bibitem{folland_course_1995}
G.~B. Folland.
\newblock {\em A course in abstract harmonic analysis}.
\newblock Studies in Advanced Mathematics. CRC Press, Boca Raton, FL, 1995.

\bibitem{folland_hardy_1982}
G.~B. Folland and E.~M. Stein.
\newblock {\em Hardy spaces on homogeneous groups}, volume~28 of {\em
  Mathematical Notes}.
\newblock Princeton University Press, Princeton, N.J., 1982.

\bibitem{gangolli_harmonic_1988}
R.~Gangolli and V.~S. Varadarajan.
\newblock {\em Harmonic analysis of spherical functions on real reductive
  groups}, volume 101 of {\em Ergebnisse der Mathematik und ihrer Grenzgebiete
  [Results in Mathematics and Related Areas]}.
\newblock Springer-Verlag, Berlin, 1988.

\bibitem{goodman_filtrations_1977}
R.~Goodman.
\newblock Filtrations and asymptotic automorphisms on nilpotent {L}ie groups.
\newblock {\em J. Differential Geometry}, 12(2):183--196, 1977.

\bibitem{goodman_nilpotent_1976}
R.~W. Goodman.
\newblock {\em Nilpotent {L}ie groups: structure and applications to analysis}.
\newblock Lecture Notes in Mathematics, Vol. 562. Springer-Verlag, Berlin,
  1976.

\bibitem{grafakos_classical_2008}
L.~Grafakos.
\newblock {\em Classical {F}ourier analysis}, volume 249 of {\em Graduate Texts
  in Mathematics}.
\newblock Springer, New York, second edition, 2008.

\bibitem{greenleaf_invariant_1969}
F.~P. Greenleaf.
\newblock {\em Invariant means on topological groups and their applications}.
\newblock Van Nostrand Mathematical Studies, No. 16. Van Nostrand Reinhold Co.,
  New York, 1969.

\bibitem{guivarch_croissance_1973}
Y.~Guivarc'h.
\newblock Croissance polynomiale et p\'eriodes des fonctions harmoniques.
\newblock {\em Bull. Soc. Math. France}, 101:333--379, 1973.

\bibitem{hebisch_smooth_1990}
W.~Hebisch and A.~Sikora.
\newblock A smooth subadditive homogeneous norm on a homogeneous group.
\newblock {\em Studia Math.}, 96(3):231--236, 1990.

\bibitem{helffer_conditions_1982}
B.~Helffer.
\newblock Conditions n\'ecessaires d'hypoanalyticit\'e pour des op\'erateurs
  invariants \`a gauche homog\`enes sur un groupe nilpotent gradu\'e.
\newblock {\em J. Differential Equations}, 44(3):460--481, 1982.

\bibitem{helffer_caracterisation_1979}
B.~Helffer and J.~Nourrigat.
\newblock Caracterisation des op\'erateurs hypoelliptiques homog\`enes
  invariants \`a gauche sur un groupe de {L}ie nilpotent gradu\'e.
\newblock {\em Comm. Partial Differential Equations}, 4(8):899--958, 1979.

\bibitem{helgason_differential_1962}
S.~Helgason.
\newblock {\em Differential geometry and symmetric spaces}.
\newblock Pure and Applied Mathematics, Vol. XII. Academic Press, New York,
  1962.

\bibitem{helgason_groups_1984}
S.~Helgason.
\newblock {\em Groups and geometric analysis}, volume 113 of {\em Pure and
  Applied Mathematics}.
\newblock Academic Press Inc., Orlando, FL, 1984.
\newblock Integral geometry, invariant differential operators, and spherical
  functions.

\bibitem{hewitt_abstract_1979}
E.~Hewitt and K.~A. Ross.
\newblock {\em Abstract harmonic analysis. {V}ol. {I}}, volume 115 of {\em
  Grundlehren der Mathematischen Wissenschaften [Fundamental Principles of
  Mathematical Sciences]}.
\newblock Springer-Verlag, Berlin, second edition, 1979.
\newblock Structure of topological groups, integration theory, group
  representations.

\bibitem{hulanicki_spectrum_1972}
A.~Hulanicki.
\newblock On the spectrum of convolution operators on groups with polynomial
  growth.
\newblock {\em Invent. Math.}, 17:135--142, 1972.

\bibitem{hulanicki_subalgebra_1974}
A.~Hulanicki.
\newblock Subalgebra of {$L_{1}(G)$} associated with {L}aplacian on a {L}ie
  group.
\newblock {\em Colloq. Math.}, 31:259--287, 1974.

\bibitem{hulanicki_commutative_1975}
A.~Hulanicki.
\newblock Commutative subalgebra of {$L^{1}(G)$} associated with a subelliptic
  operator on a {L}ie group {$G$}.
\newblock {\em Bull. Amer. Math. Soc.}, 81:121--124, 1975.

\bibitem{jenkins_dilations_1979}
J.~W. Jenkins.
\newblock Dilations and gauges on nilpotent {L}ie groups.
\newblock {\em Colloq. Math.}, 41(1):95--101, 1979.

\bibitem{loomis_introduction_1953}
L.~H. Loomis.
\newblock {\em An introduction to abstract harmonic analysis}.
\newblock D. Van Nostrand Company, Inc., Toronto-New York-London, 1953.

\bibitem{ludwig_sub-laplacians_2000}
J.~Ludwig and D.~M{\"u}ller.
\newblock Sub-{L}aplacians of holomorphic {$L^p$}-type on rank one
  {$AN$}-groups and related solvable groups.
\newblock {\em J. Funct. Anal.}, 170(2):366--427, 2000.

\bibitem{martini_multipliers_2010}
A.~Martini.
\newblock {\em {Algebras of differential operators on Lie groups and spectral
  multipliers}}.
\newblock {Tesi di perfezionamento (PhD thesis)}, Scuola Normale Superiore,
  Pisa, 2010.
\newblock \texttt{arXiv:1007.1119}.

\bibitem{maurin_general_1968}
K.~Maurin.
\newblock {\em General eigenfunction expansions and unitary representations of
  topological groups}.
\newblock Monografie Matematyczne, Tom 48. PWN-Polish Scientific Publishers,
  Warsaw, 1968.

\bibitem{megginson_introduction_1998}
R.~E. Megginson.
\newblock {\em An introduction to {B}anach space theory}, volume 183 of {\em
  Graduate Texts in Mathematics}.
\newblock Springer-Verlag, New York, 1998.

\bibitem{miller_parametrices_1980}
K.~G. Miller.
\newblock Parametrices for hypoelliptic operators on step two nilpotent {L}ie
  groups.
\newblock {\em Comm. Partial Differential Equations}, 5(11):1153--1184, 1980.

\bibitem{nagel_balls_1985}
A.~Nagel, E.~M. Stein, and S.~Wainger.
\newblock Balls and metrics defined by vector fields. {I}. {B}asic properties.
\newblock {\em Acta Math.}, 155(1-2):103--147, 1985.

\bibitem{nelson_representation_1959}
E.~Nelson and W.~F. Stinespring.
\newblock Representation of elliptic operators in an enveloping algebra.
\newblock {\em Amer. J. Math.}, 81:547--560, 1959.

\bibitem{palmer_classes_1978}
T.~W. Palmer.
\newblock Classes of nonabelian, noncompact, locally compact groups.
\newblock {\em Rocky Mountain J. Math.}, 8(4):683--741, 1978.

\bibitem{palmer_banach_2001}
T.~W. Palmer.
\newblock {\em Banach algebras and the general theory of {$*$}-algebras. {V}ol.
  2}, volume~79 of {\em Encyclopedia of Mathematics and its Applications}.
\newblock Cambridge University Press, Cambridge, 2001.
\newblock $*$-algebras.

\bibitem{rudin_functional_1973}
W.~Rudin.
\newblock {\em Functional analysis}.
\newblock McGraw-Hill Book Co., New York, 1973.
\newblock McGraw-Hill Series in Higher Mathematics.

\bibitem{rudin_real_1974}
W.~Rudin.
\newblock {\em Real and complex analysis}.
\newblock McGraw-Hill Book Co., New York, second edition, 1974.
\newblock McGraw-Hill Series in Higher Mathematics.

\bibitem{ter_elst_weighted_1994}
A.~F.~M. ter Elst and D.~W. Robinson.
\newblock Weighted strongly elliptic operators on {L}ie groups.
\newblock {\em J. Funct. Anal.}, 125(2):548--603, 1994.

\bibitem{ter_elst_spectral_1997}
A.~F.~M. ter Elst and D.~W. Robinson.
\newblock Spectral estimates for positive {R}ockland operators.
\newblock In {\em Algebraic groups and {L}ie groups}, volume~9 of {\em Austral.
  Math. Soc. Lect. Ser.}, pages 195--213. Cambridge Univ. Press, Cambridge,
  1997.

\bibitem{ter_elst_weighted_1998}
A.~F.~M. ter Elst and D.~W. Robinson.
\newblock Weighted subcoercive operators on {L}ie groups.
\newblock {\em J. Funct. Anal.}, 157(1):88--163, 1998.

\bibitem{thomas_infinitesimal_1984}
E.~G.~F. Thomas.
\newblock An infinitesimal characterization of {G}el\cprime fand pairs.
\newblock In {\em Conference in modern analysis and probability ({N}ew {H}aven,
  {C}onn., 1982)}, volume~26 of {\em Contemp. Math.}, pages 379--385. Amer.
  Math. Soc., Providence, RI, 1984.

\bibitem{treves_topological_1967}
F.~Tr{\`e}ves.
\newblock {\em Topological vector spaces, distributions and kernels}.
\newblock Academic Press, New York, 1967.

\bibitem{varadarajan_lie_1974}
V.~S. Varadarajan.
\newblock {\em Lie groups, {L}ie algebras, and their representations}.
\newblock Prentice-Hall Inc., Englewood Cliffs, N.J., 1974.
\newblock Prentice-Hall Series in Modern Analysis.

\bibitem{varopoulos_hardy_1995}
N.~T. Varopoulos.
\newblock Hardy-{L}ittlewood theory on unimodular groups.
\newblock {\em Ann. Inst. H. Poincar\'e Probab. Statist.}, 31(4):669--688,
  1995.

\bibitem{varopoulos_analysis_1992}
N.~T. Varopoulos, L.~Saloff-Coste, and T.~Coulhon.
\newblock {\em Analysis and geometry on groups}, volume 100 of {\em Cambridge
  Tracts in Mathematics}.
\newblock Cambridge University Press, Cambridge, 1992.

\bibitem{vinberg_commutative_2001}
{\`E}.~B. Vinberg.
\newblock Commutative homogeneous spaces and co-isotropic symplectic actions.
\newblock {\em Uspekhi Mat. Nauk}, 56(1(337)):3--62, 2001.

\bibitem{wolf_harmonic_2007}
J.~A. Wolf.
\newblock {\em Harmonic analysis on commutative spaces}, volume 142 of {\em
  Mathematical Surveys and Monographs}.
\newblock American Mathematical Society, Providence, RI, 2007.

\end{thebibliography}

\end{document}